\documentclass[11pt,reqno]{amsart}

\RequirePackage{tikz}
\RequirePackage{tikz-cd}

\usepackage[T1]{fontenc}

\usepackage{amsmath,xypic}								
\usepackage{amssymb}
\usepackage{amsthm}
\usepackage{amscd}
\usepackage{amsfonts}
\usepackage{bm}
\usepackage{stmaryrd}
\usepackage[all]{xy}

\usepackage{euler}

\usepackage{extarrows}

\usepackage{color}

\usetikzlibrary{matrix}
\usetikzlibrary{patterns}
\usetikzlibrary{matrix}
\usetikzlibrary{positioning}
\usetikzlibrary{decorations.pathmorphing}

\usepackage{enumerate}

\newtheorem{maintheorem}{Theorem}
\newtheorem{maindefi}{Definition}										
\newtheorem{theorem}{Theorem}[section]
\newtheorem{lemma}[theorem]{Lemma}
\newtheorem{proposition}[theorem]{Proposition}
\newtheorem{corollary}[theorem]{Corollary} 
 
\newtheorem*{thmdo}{Theorem \ref{diffop}}
\theoremstyle{definition}
\newtheorem{definition}[theorem]{Definition}

\newtheorem{remark}[theorem]{Remark}

\newcommand{\calF}{\mathcal{F}}
\newcommand{\calG}{\mathcal{G}}

\newcommand{\calK}{\mathcal{K}}
\newcommand{\calL}{\mathcal{L}}
\newcommand{\calM}{\mathcal{M}}

\newcommand{\calP}{\mathcal{P}}

\newcommand{\calS}{\mathcal{S}}

\newcommand{\bZ}{\mathbb{Z}}
\newcommand{\bH}{\mathbb{H}}

\newcommand{\DD}{\displaystyle}
\newcommand{\vardel}{\partial}

\newcommand{\Moduli}{\overline{\mathcal{M}}}
\newcommand{\UF}{\overline{\mathcal{U}}}

\renewcommand{\tilde}{\widetilde}

\title[]{Witten's Conjecture and Recursions for $\kappa$ Classes }

\author{Vance Blankers}
\address{Department of Mathematics, Colorado State University, Fort Collins, Colorado 80523-1874}
\email{{blankers@mail.colostate.edu }}
\thanks{}

\author{Renzo Cavalieri }
\address{Department of Mathematics, Colorado State University, Fort Collins, Colorado 80523-1874}
\email{{renzo@math.colostate.edu}}
\thanks{R.C. acknowledges support from Simons Foundation Collaboration Grant 420720}

\subjclass[2010]{14N10,14N35}

\begin{document}
\allowdisplaybreaks

\begin{abstract} 
We construct a countable number of differential operators $\hat{L}_n$ that annihilate a generating function  for  intersection numbers of $\kappa$ classes on $\Moduli_g$ (the $\kappa$-potential).
This produces recursions among intersection numbers of $\kappa$ classes which determine all such numbers  from a single initial condition.
The starting point of the work is a combinatorial formula relating intersecion numbers of $\psi$ and $\kappa$ classes. Such a formula produces an exponential differential operator acting on the Gromov-Witten potential to produce the $\kappa$-potential; after restricting to a hyperplane, we have an explicit change of variables relating the two generating functions, and we conjugate the ``classical'' Virasoro operators to obtain the operators $\hat{L}_n$.
\end{abstract}
\maketitle






\section*{Introduction}




\subsection*{Main result}
We state immediately  the main result of this manuscript, and refer the reader to the next section in the introduction for a more leisurely discussion leading to it.
\begin{maindefi} \label{kavira}
For $n= 0,1$,  $\hat{L}_n$ denotes the  differential operator:

\begin{align}
\hat{L}_{0}  &= -\frac{3}{2} \vardel_{{p}_0}+  \sum_{m=0}^\infty  mp_m\vardel_{{p}_m} +\frac{1}{16}, \nonumber\\
\hat{L}_{1}  & = -\frac{15}{4} \vardel_{{p}_1} +    \sum_{m=0}^\infty  m(m+4)p_m\vardel_{{p}_{m+1}} -  \sum_{l,m=0}^\infty  lmp_lp_m\vardel_{{p}_{m+l+1}} \nonumber\\
 & + \frac{(ue^{p_0})^2 }{8} \left( \sum_{m=0}^\infty \left(S_{m+2}({\mathbf{p}})- S_{m+2}(2{\mathbf{p}}) \right)  \vardel_{ {p}_m}
  + \sum_{l,m=0}^\infty   S_{l+1}({\mathbf{p}}) S_{m+1}({\mathbf{p}}) \vardel_{{p}_l} \vardel_{{p}_m}\right); \nonumber 
 \end{align} 
   for all  $n\geq 2$,
 \begin{align} 
  \hat{L}_n &= \sum_{d=0}^{n+1} \alpha_{n,d} \left( -\sum_{m=0}^\infty \left[B_d (q_1, \ldots, q_d)\right]_{z^m}\vardel_{{p}_{m+n}} \right) \nonumber \\
  & +\frac{(ue^{p_0})^2}{2}\left[\left(\sum_{i=0}^{n-1}\frac{(2i+1)!!(2n-2i-1)!!}{2^{n+1}}\right) \left(\sum_{m=0}^\infty    S_{m}(2{\mathbf{p}})  \vardel_{{p}_{m+n-3}}\right)\right. \nonumber\\
 & + \left.\sum_{i=0}^{n-1}\frac{(2i+1)!!(2n-2i-1)!!}{2^{n+1}} \sum_{m,l=0}^\infty S_{m}({\mathbf{p}})  S_{l}({\mathbf{p}}) \vardel_{ {p}_{m+n-2-i}} \vardel_{ {p}_{l+i-1}}  \right],
\end{align}
where $\mathbf{p} = (p_1, p_2, p_3, \ldots)$ denotes a countable list of variables, $S_i$ is the $i$-th elementary Schur polynomial (Definition \ref{def:sch}), $B_d$ denotes the $d$-th Bell polynomial (Definition \ref{def:bell}); the symbols $q_i = - \sum_{k} k^i p_kz^k$ and 
$\alpha_{n,d}  = \left[\prod_{i=0}^n \left(x+i+\frac{3}{2}\right)\right]_{x^d}$.
\end{maindefi}

Let $\calK(p_0, \mathbf{p})$  denote a generating function for top intersection numbers of $\kappa$ classes (Definition \ref{def:kapot}).

\begin{maintheorem}\label{mainthm}
For all $n\geq 0$, we have $$\hat{L}_n(e^\calK)=0.$$
 The vanishing of the coefficients of  monomials in $\hat{L}_n(e^\calK)$ gives  recursive relations among  intersection numbers of $\kappa$ classes; the collection  of all such recursions uniquely determines $\calK$ from the initial condition $\vardel_{p_0}\calK_{|(p_0,\mathbf{p})=(0,\mathbf{0})} = 1/24$.
\end{maintheorem}

\subsection*{Context and motivation}

The study of the tautological intersection theory of the moduli space of curves was initiated in the seminal \cite{m:taegotmsoc}. In Mumford's words, this entails
\begin{quote}{\it
[...] setting up a Chow ring for the moduli space $\calM_g$ of curves of genus $g$ and its compactification $\Moduli_g$, defining what seem to be the most important classes in this ring [...]}
\end{quote}

 Besides boundary strata, two families of classes   played a prominent role: $\psi$ classes (Definition \ref{def-psi}) and $\kappa$ classes (Definition \ref{def-kappa}). Geometrically, $\psi$ classes arise when taking non-tranversal intersections of boundary strata. The most remarkable feature of intersection numbers of $\psi$ classes is Witten's conjecture/Kontsevich's theorem (\cite{ wittenconj,k:wittconj}): a generating function $\calF$ for intersection numbers of $\psi$ classes (the Gromov-Witten potential) is a $\tau$ function for the KdV hierarchy. One formulation of this statement says there exist a countable number of differential operators $L_n$ ($n\geq -1$) that annihilate $e^\calF$.  The vanishing of the coefficient of each monomial in $e^\calF$ gives a recursion on the intersection numbers of $\psi$ classes, and such recursions allow computations of all top intersection numbers of $\psi$ classes from the initial condition $\int_{\Moduli_{0,3}} 1 = 1$. 
 
 Mumford's conjecture (\cite{m:taegotmsoc}), now a theorem by Madsen and Weiss (\cite{mw:mg}), brought $\kappa$ classes to the foreground:  the stable cohomology of $\calM_g$ is a polynomial ring freely generated by the $\kappa$ classes.
 More recently,  Pandharipande (\cite{panda:k})  studied the restriction of the $\kappa$-ring,  the part of the tautological ring of the moduli space of curves generated by $\kappa$ classes, to the locus of curves of compact type; he unveiled the following interesting structure: the higher genus $\kappa$-rings are quotients of the genus zero ones in a canonical way.
   
 The intersection theories of $\psi$ and $\kappa$ classes both exhibit rich combinatorial structure and are born out of operations involving Chern classes of the relative dualizing sheaf and tautological morphisms. It should not be completely surprising that these two theories are closely related.
 
It is mentioned in  \cite{ac:cag}, and there credited to Carel Faber, that the degree of a top intersection monomial in $\psi$ classes on $\Moduli_{g,n} $ may be expressed as a polynomial in $\kappa$ classes on $\Moduli_g$, explicitly described as a sum of monomials indexed by elements in the symmetric group $S_n$. In \cite{wittenconj}, Witten remarks that knowing the intersection numbers of $\kappa$ classes on $\Moduli_g$ is equivalent to knowing intersection numbers of monomials in $\psi$ classes on all the $\Moduli_{g,n}$. An explicit formula to express a monomial in $\kappa$ classes as a polynomial in $\psi$ classes is derived independently in \cite{kmz}, \cite{alexeev-guy} and \cite{BC:omega} (see Corollary \ref{num}) .
 
Manin and Zograf (\cite[Thm 4.1]{mz}) show  the  formulae to express a single monomial in $\kappa$ classes on $\Moduli_g$ as a polynomial in $\psi$ classes on different $\Moduli_{g,n}$  directly imply that generating functions encoding intersection numbers of $\kappa$ (the $\kappa$-potential) and $\psi$ classes  are related by an explicitly described change of variables.

Our first result  is a new proof of Manin and Zograf's result, introducing a few intermediate steps which make the proof more conceptual and easy to understand. First, the formulae \eqref{omegafor2} allow us to construct an operator $\calL$   (see \eqref{fork}), such that $e^\calL$ acts on (the $g\geq 1$ part of) the Gromov-Witten potential to produce a potential  for $\omega$ classes (Definition \ref{def-omega}).

\begin{thmdo}
Denote by $\tilde\calF(\mathbf{t})$ the positive genus part of the Gromov-Witten potential, and by $\calS(\mathbf{s})$ the $\omega$-potential. We exhibit an operator $\calL$ such that:
\begin{equation*}
(e^\calL \tilde{\calF})_{|\mathbf{t}=0} = \calS.
\end{equation*}
\end{thmdo}

We have an explicit expression for the operator $\calL = \sum f_n(\mathbf{s}) \vardel_{t_n}$ as a vector field.  It follows 
that the two potentials are related by a change of variables given by the coefficients $f_n$ (Lemma \ref{expflow}). 

It is  a simple consequence of projection formula that the $\omega$-potential $\calS(\mathbf{s})$ and $\kappa$-potential $\calK(\mathbf{p})$ are related by specializing the variable $s_0=0$ and shifting down the other variables, so that $p_{n-1} = s_n$ (Lemma \ref{l:KSS}).

Combining these results, Theorem \ref{thm:mz} is a reformulation of Manin and Zograf's Theorem 4.1 (\cite{mz}), restricted to intersections of $\kappa$ classes on moduli spaces of unpointed curves $\Moduli_g$.

In the process of proving Theorem  \ref{thm:mz}, we note that we are restricting the domain of the Gromov-Witten potential to the hyperplane $t_0=0$. 
After dilaton shift and some cosmetic change of variables mostly  involving shifts and signs, this simplifies the change of variables relating $\calF_{|t_0 = 0}$ and $\calK$ to the following expression (where the coefficient of each power of $z$ identifies a variable $\hat{t}_{i}$ as a function of the variables $\hat{p}_{k} $):
\begin{equation} \label{incov}
\sum_{i=0}^\infty \hat{t}_{i} z^i = e^{\sum_{k=0}^\infty \hat{p}_{k} z^{k}}. 
\end{equation}

We then turn our attention to the Virasoro operators. The idea is simple: rewriting the differential operators $L_n$ in the variables for the $\kappa$-potential should produce a hierarchy of operators annihilating $e^\calK$, thereby producing relations among intersection numbers of $\kappa$ classes. In order to carry this out we use the string equation $L_{-1}$ to remove any occurrence of $\vardel_{t_0}$ from the $L_n$s. We then exploit the structure of the Jacobian of the change of variables \eqref{incov} -- which can be thought of as an infinite upper triangular matrix constant along translates of the diagonal -- to reduce performing the chain rule to manipulation operations on generating series. 

While Theorem \ref{mainthm} is shown as a consequence of Witten's conjecture, in fact Witten's conjecture is equivalent to Theorem \ref{mainthm} plus the string equation $L_{-1}$. The string equation has an elementary geometric proof, following from the behavior of $\psi$ classes when pulled-back via forgetful morphisms (Lemma \ref{lem-comp}). Therefore an independent proof of the relations among $\kappa$ classes  arising from Theorem \ref{mainthm} would in particular give another proof of Witten's conjecture.  

This work fits in the broader context of wall-crossings among Gromov-Witten potentials of Hassett spaces of weighted stable curves: the intersection theory of $\kappa$ classes  arises as a limit when the weights of some of the points are approaching $0$.  By using this language, and considering descendant invariants on heavy-light Hassett spaces one could recover the full statement  of Manin-Zograf's theorem, which compares the Gromov-Witten potential to the $\kappa$-potential for all $\Moduli_{g,n}$; in this work we chose to restrict our attention to the special case where all points are light, both because it highlights the most essential combinatorial relation between $\psi$ and $\kappa$ classes and because it allows us to reduce the bookkeeping and tell a compelling story  within the classical context of  intersection theory  on $\Moduli_{g,n}$.  The more general picture  for Hassett spaces will be treated in \cite{bc:tocome}.

\subsection*{Organization of the paper and communication}
While the ultimate motivation of this work is to further our understanding of the geometry of moduli spaces of curves $\Moduli_{g,n}$, the techniques used are combinatorial and analytic in nature. We are making a conscious effort to communicate to the union, rather than the intersection, of these three mathematical communities. Section \ref{bp} gives a review of all relevant definitions of tautological classes on $\Moduli_{g,n}$, as well as their behavior under pull-back/push-forward via tautological morphisms. Since some of these statements are well-known to the experts but somewhat hard to track down in the literature, we include several sketches of proofs. Section \ref{sec:gf}  introduces the generating functions for $\psi$, $\omega$, and $\kappa$ classes and translates the combinatorial formula \eqref{num} relating intersection numbers of $\psi$ and $\omega$ classes to a differential operator acting on the Gromov-Witten potential. In Section \ref{sec:cov} we use this interpretation to deduce that $\calF$ and $\calK$ are related by a specialization/change of variables and write explicitly such a transformation. Section \ref{sec:chain} derives the operators $\hat{L}_n$ in Definition \ref{kavira}. Since such a derivation is rather technical and bookkeeping intensive, we break it down into several subsections, with the intention that it should be possible for a reader interested in only a part of the computation to easily isolate it. The proof of Theorem \ref{mainthm} is then an immediate consequence of the construction of the operators $\hat{L}_n$.
Finally, in Section \ref{relations} we write down some of the relations among $\kappa$ classes that are produced with this method.

\subsection*{Acknowledgements}
We would like to thank Nick Ercolani, Hiroshi Iritani for helpful conversations. R.C. acknowledges support from the Simons foundation (Collaboration Grant 420720).

\section{Background and Preliminaries}
\label{bp}


This section contains basic information about tautological classes on moduli spaces of curves and their behavior with respect to pull-backs and push-forward via tautological morphisms. While we expect most readers to have some familiarity with these topics (otherwise excellent places to start are \cite{hm:mc,v:mscgw}), we intend this section to establish notation and to highlight concepts that are relevant to this work.

Given two non-negative integers $g,n$ satisfying $2g-2+n >0$, we denote by $\Moduli_{g,n}$ the fine moduli space for families of Deligne-Mumford stable curves of genus $g$ with $n$ marked points. The space $\Moduli_{g,n}$ is a smooth and projective DM stack of dimension $3g-3+n$ \cite{k:pmsII,k:pmsIII}. We recall the tautological morphisms that provide connections among different moduli spaces. 

Let $g_1, g_2$ be two non-negative integers adding to $g$, and $(P_1, P_2)$  a partition\footnote{If $g_i = 0$, we require $|P_i| \geq 2$} of the set $[n] = \{1,\dots,n\}$. The {\bf gluing morphism}
\begin{equation}
gl_{(g_1,P_1)| (g_2, P_2)}: \Moduli_{g_1, P_1\cup \{\bullet\}} \times  \Moduli_{g_2, P_2\cup \{\star\}}\to  \Moduli_{g, n}
\end{equation} 
assigns to a pair $((C_1; \{p_i\}_{i\in P_1}, \bullet), (C_2; \{p_j\}_{j\in P_2}, \star))$ the pointed nodal curve $(C_1\cup_{ \bullet = \star}C_2; p_1, \ldots, p_n)$ obtained by identifying the marks denoted by $\star$ and $\bullet$.  The image of the gluing morphism is an irreducible, closed subvariety of $\Moduli_{g,n}$ called a {\bf boundary divisor} and denoted $D((g_1,P_1)|(g_2, P_2))$.

More generally, given a nodal, pointed, stable curve $(C; p_1, \ldots, p_n)$, we can identify its topological type by its dual graph $\Gamma$ (\cite{y:inom}).  We define a more general gluing morphism
\begin{equation}
gl_\Gamma: \prod_{v\in V(\Gamma)} \Moduli_{g(v), val(v)} \to \Moduli_{g,n}
\end{equation}  
to be the morphism that glues marked points corresponding to pairs of flags that form an edge of the dual graph. The morphism $gl_\Gamma$ is finite of degree $|Aut(\Gamma)|$ onto its image. We call the image of $gl_\Gamma$ a {\bf boundary stratum} and denote it by $\Delta_\Gamma$. As a cycle class:
\begin{equation}\label{bs}
[\Delta_\Gamma] = \frac{1}{|Aut(\Gamma)|} gl_{\Gamma \ast}(1).
\end{equation}

For any mark $i \in [n+1]$, there is a {\bf forgetful morphism}
\begin{equation}
\pi_i: \Moduli_{g,n+1} \to \Moduli_{g, [n+1]\smallsetminus \{i\}} \cong \Moduli_{g, n},
\end{equation}
which assigns to an $(n+1)$-pointed curve $(C;p_1, \ldots, p_{n+1})$ the $n$-pointed curve obtained by forgetting the $i$-th marked point and  contracting any rational component of $C$ which has less than three special points (marks or nodes). The morphism $\pi_i$ functions as a universal family  for $\Moduli_{g,n}$, and so in particular the universal curve  $\UF_{g,n} \to \Moduli_{g,n}$ may be identified with $\Moduli_{g,n+1}$. 

The {\bf $i$-th tautological section} 
\begin{equation}
\sigma_i: \Moduli_{g,n} \to \UF_{g,n} \cong \Moduli_{g,n+1}
\end{equation}
assigns to an $n$-pointed curve $(C;p_1, \ldots, p_n)$ the point $p_i$ in the fiber over $(C;p_1, \ldots, p_n)$ in the universal curve. Such a point corresponds to the $(n+1)$-pointed curve obtained by attaching a rational component to the point $p_i\in C$ and placing the marks $p_i$ and $p_{n+1}$ arbitrarily on the new rational component. Via the identification of the universal map with a forgetful morphism, the section $\sigma_i$ may be viewed as a gluing morphism and its image as a boundary stratum, denoted $\Delta_{i,n+1}$. The  following diagram  illustrates this concept:
\begin{equation}
\xymatrix{
\UF_{g,n} \ar[rr]^{\cong} & & \Moduli_{g, n+1}\\
Im(\sigma_i) \ar@{}[u]|-*[@]{\subseteq}& & \Delta_{i,n+1}\ar@{}[u]|-*[@]{\subseteq}\\
\Moduli_{g,n} \ar[rr]^{\hspace{-1cm} \cong}  \ar[u]^{\sigma_i}&  & \Moduli_{g, [n]\smallsetminus \{i\} \cup \{\bullet\}} \times \Moduli_{0, \{\star, i, n+1\}} \ar[u]_{gl_{((g, [n]\smallsetminus \{i\} \cup \{\bullet\})|(0,\{\star, i, n+1\}))}}
}
\end{equation}

We consider all $\Moduli_{g,n}$ (for all  values of $g,n$) as a system of moduli spaces connected by the tautological morphisms and define the {\bf tautological ring} $\mathcal{R} = \{R^\ast(\Moduli_{g,n})\}_{g,n}$ of this system to be the smallest system of subrings of the Chow ring of  each $\Moduli_{g,n}$ containing all fundamental classes $[\Moduli_{g,n}]$ and closed under push-forwards and pull-backs via the tautological (gluing and forgetful) morphisms. By \eqref{bs}, classes of boundary strata are elements of the tautological ring.

We now introduce some other families of tautological classes which are studied in this work.

\begin{definition}\label{def-psi}
For any choice of mark $i \in [n]$, the class $\psi_i \in R^1(\Moduli_{g,n})$ is defined to be:
\begin{equation}\psi_i := c_1(\sigma_i^\ast(\omega_{\pi})),\end{equation} where $\omega_{\pi}$ denotes the relative dualizing sheaf of the universal family $\pi: \UF_{g,n} \to \Moduli_{g,n}$.
\end{definition}

\begin{definition}\label{def-omega}
Let $g,n\geq 1$, $i\in [n]$, and let $\rho_i:\Moduli_{g,n}\to\Moduli_{g,\{i\}}$ be the composition of forgetful morphisms for  all but the $i$-th mark. Then we define 
\begin{equation}
\omega_i := \rho_i^*\psi_i
\end{equation}
 in $R^1(\Moduli_{g,n}$). 
\end{definition}

\begin{definition}\label{def-kappa}
For a non-negative integer $i$, the class $\kappa_i \in R^i(\Moduli_{g,n})$ is:
\begin{equation}\kappa_i := \pi_{n+1 \ast}(\psi_{n+1}^{i+1}).\end{equation} 
\end{definition}

\begin{remark}
In \cite{m:taegotmsoc}, Mumford introduces first $\kappa$ classes on $\Moduli_{g}$ as
\begin{equation}
\kappa_i:= \pi_\ast(c_1(\omega_\pi)^{i+1}),
\end{equation}
where $\pi: \Moduli_{g,1} \cong \UF_{g} \to \Moduli_{g}$ denotes  the universal family. One may verify that on $\Moduli_{g,1}$ 
\begin{equation}
\psi_1 = c_1(\omega_\pi),
\end{equation}
which makes the (unique) $\psi$ class on $\Moduli_{g,1}$ an uncontroversially natural and canonically constructed class. There are two different ways to generalize this class to spaces with more than one mark: simply pulling-back  the class  $\psi$  gives rise to the class $\omega$, whereas focusing on the fact that $\psi$ is the Euler class of a line bundle whose fiber over the moduli point  $(C;p)$ is $T^\ast_p(C)$ generalizes to the definition of $\psi$ class given above.
\end{remark}

The following lemma shows how $\psi$ classes behave when pulled-back via forgetful morphisms.

\begin{lemma}\cite{k:pc} \label{lem-comp}
Consider the forgetful morphism $\pi_{n+1}: \Moduli_{g,n+1} \to \Moduli_{g,n}$, and let the context determine whether $\psi_i$ denotes the class on $\Moduli_{g,n}$ or $\Moduli_{g,n+1}$. For $i\in [n]$, we have:
\begin{equation} \label{pbr}
\psi_i = \pi_{n+1}^\ast (\psi_i) + D_{i,n+1},
\end{equation}
where  $D_{i,n+1}$ denotes the class of the image of the section $\sigma_i$, or equivalently the class of the boundary divisor $\Delta_{i,n+1}$, generically parameterizing nodal curves where one component  is rational and hosts the $i$-th and $(n+1)$-th marks.
\end{lemma}

Iterated applications of Lemma \ref{lem-comp} show the relation between the classes $\psi_i$ and $\omega_i$ on $\Moduli_{g,n}$. Denote by $D(A|B)$ the divisor $D((g,A)|(0,B))$. We call any boundary stratum where all the genus is concentrated at one vertex of the dual graph a stratum of {\bf rational tails} type.

\begin{lemma}
Let $g,n\geq 1$ and $i\in [n]$. Then:
\begin{equation} \label{ompsirel}
\psi_i= \omega_i + \sum_{B\ni i} D(A|B)
\end{equation}
\end{lemma}
In words, this means that $\psi_i$ is obtained from $\omega_i$ by adding all divisors of rational tails type where the $i$-th mark is contained in the rational component.

The behavior of $\kappa$ classes under pull-back via forgetful morphisms  was studied in \cite{ac:cag}.
\begin{lemma}[\cite{ac:cag}]
Consider the forgetful morphism $\pi_{n+1}: \Moduli_{g,n+1} \to \Moduli_{g,n}$, and let the context determine whether $\kappa_i$ denotes the class on $\Moduli_{g,n}$ or $\Moduli_{g,n+1}$. We have:
\begin{equation}
\kappa_i = \pi_{n+1}^\ast(\kappa_i)+ \psi_{n+1}^i
\end{equation}
\end{lemma}

The next group of Lemmas gives information about the behavior of tautological classes under push-forward via forgetful morphisms. These are familiar facts for people in the field, but it is non-trivial to track down appropriate references; for this reason we add brief sketches of proofs that could be completed by the interested reader.
Pushing-forward a monomial in $\psi$ classes along a morphism that forgets a mark that does not support a $\psi$ class one obtains the so-called {\bf string recursion}. 
\begin{lemma}[String]\label{str}
Consider the forgetful morphism $\pi_{n+1}: \Moduli_{g,n+1} \to \Moduli_{g,n}$. Let $K \in \bZ^n$ denote the vector  $(k_1, \ldots, k_n)$ and $e_i$  the $i$-th standard basis vector. By $\psi^K$ we mean $\prod \psi_i^{k_i}$ and we adopt the convention that $\psi_i^m = 0$ whenever $m<0$. Then:
\begin{equation}\label{se}
\pi_{n+1 \ast}(\psi^K) = \sum_{i=1}^n \psi^{K-e_i}. 
\end{equation}
\end{lemma}
\begin{proof}
Equation \eqref{se} is proved  by using Lemma \ref{lem-comp} to replace each $\psi_i^{k_i}$ with $\pi_{n+1}^\ast( \psi_i)^{k_i} + D_{i,n+1}\pi_{n+1}^\ast (\psi_i)^{k_i-1}$, and then applying projection formula to obtain the right hand side of \eqref{se}. More details can be found in \cite[Lemma 1.4.2]{k:pc}.
\end{proof}

The analogous statement for $\omega$ classes is rather trivial.
\begin{lemma}[$\omega$-string] \label{ostring}
Let $n>0$, and consider the forgetful morphism $\pi_{n+1}: \Moduli_{g,n+1} \to \Moduli_{g,n}$. Let $K \in \bZ^n$ denote the vector  $(k_1, \ldots, k_n)$ and $\omega^K = \prod \omega_i^{k_i}$. Then:
\begin{equation}
\pi_{n+1 \ast}(\omega^K) = 0.
\end{equation}
\end{lemma}
\begin{proof}
The map $\pi_{n+1}$ has positive dimensional fibers, and by definition $\omega^K = \pi_{n+1}^\ast(\omega^K)$, which implies the vanishing of the push-forward.
\end{proof}

\begin{lemma}\label{ka0}
For any $g,n$ with $2g-2+n>0$, 
\begin{equation}
\kappa_0 = (2g-2+n) [1]_{\Moduli_{g,n}}.
\end{equation}
\end{lemma}
\begin{proof}
The  space $\Moduli_{g,n}$ is proper, connected, and irreducible, hence the class $\kappa_0$ must be a multiple of the fundamental class. Fixing a moduli point $[(C; p_1, \ldots, p_n)]\in \Moduli_{g,n}$, by projection formula the multiple is computed as $deg(\omega_C(p_1+ \ldots+ p_n)) = 2g-2+n$.
\end{proof}

\begin{lemma}[Dilaton] \label{dil}
Consider the forgetful morphism $\pi_{n+1}: \Moduli_{g,n+1} \to \Moduli_{g,n}$. Let $K \in \bZ^n$ denote a vector of non-negative integers  $(k_1, \ldots, k_n)\not= (0,\ldots, 0)$.  Then:
\begin{equation}
\pi_{n+1 \ast}(\psi^K\psi_{n+1}) = (2g-2+n) \psi^{K}. 
\end{equation}
\begin{proof}
By Lemma \ref{lem-comp} and the basic vanishing $\psi_{n+1} D_{i,n+1}=0$, we have
$$
\psi^K\psi_{n+1} = \pi_{n+1}^\ast (\psi^K)\psi_{n+1}.
$$
The proof is concluded by applying projection formula and using Lemma \ref{ka0}.
\end{proof}

\end{lemma}

\begin{lemma}[$\omega$-dilaton]\label{odil}
Let $g+n\geq 2$, and consider the forgetful morphism $\pi_{n+1}: \Moduli_{g,n+1} \to \Moduli_{g,n}$. Let $K \in \bZ^n$ denote a vector of non-negative integers  $(k_1, \ldots, k_n)$.  Then:
\begin{equation}\label{se2}
\pi_{n+1 \ast}(\omega^K\omega_{n+1}) = (2g-2) \omega^{K}. 
\end{equation}
\end{lemma}
\begin{proof}
First assume $g\geq 2$, and consider the following commutative diagram:
\begin{equation}\label{diapb}
\xymatrix{
\Moduli_{g, n+1} \ar[r]^{\pi_{n+1}} \ar[d]_{\rho_{n+1}} &\ar[d]^F \Moduli_{g, n}\\ 
\Moduli_{g, \{n+1\}} \ar[r]^{\pi_{n+1}} & \Moduli_{g}.
}
\end{equation}
We have $\pi_{n+1 \ast}(\omega_{n+1}) =\pi_{n+1 \ast}\rho_{n+1}^\ast (\psi_{n+1}) = F^\ast \pi_{n+1 \ast} (\psi_{n+1})= (2g-2) [1]_{\Moduli_{g,n}} $. Equation \eqref{se2} then follows by projection formula.

When $g=1$, any monomial in $\omega$ classes of degree greater than one vanishes because $\omega_i = \lambda_1$ and $\lambda_1^2 = 0$ (\cite[(5.4)]{m:taegotmsoc}). Similarly, $\pi_{n+1\ast}(\omega_i) = 0$, since $\lambda_1$ is pulled-back from $\Moduli_{1,1}$.
\end{proof}

The proof of Lemma \ref{odil} generalizes to give a natural relation between $\omega$ and $\kappa$ classes.
\begin{lemma}\cite[Lem. 3.3]{BC:omega}\label{oka}
Let $g\geq 2$, and consider the total forgetful morphism $F: \Moduli_{g,n} \to \Moduli_{g}$. Let $K \in \bZ^n$ denote a vector of non-negative integers  $(k_1, \ldots, k_n)$ and $\vec{1} = (1,1, \ldots, 1)$.  Then:
\begin{equation}
F_{ \ast}(\omega^K) =  \kappa_{K-\vec{1}}. 
\end{equation}

\end{lemma}

In \cite{BC:omega} a graph formula is used to express an arbitrary monomial in $\omega$ classes in terms of a special class of boundary strata appropriately decorated with $\psi$ classes.  Denote by $\calP \vdash [n]$ an unordered partition of the set $[n]$, i.e. a collection of  pairwise disjoint non-empty subsets $P_1,\dots, P_r$ such that 
$$
P_1 \cup \ldots \cup P_r = [n].
$$
Given $\calP \vdash [n]$, when $|P_i| = 1$  denote by $\bullet_i$ the element of the singleton $P_i$.  For $|P_i|>1$, introduce new labels $\bullet_i$ and $\star_i$. The \emph{pinwheel stratum} $\Delta_{\calP}$ is the image of the gluing morphism
$$
gl_{\calP}: \Moduli_{g, \{\bullet_1, \ldots, \bullet_r\}} \times \prod_{|P_i|>1} \Moduli_{0, \{\star_i\}\cup P_i} \to \Moduli_{g,n}
$$
that glues together each $\bullet_i$ with $\star_i$. The class of the stratum equals the push-forward of the fundamental class via $gl_{\calP}$.

\begin{proposition}\cite[Thm. 2.2]{BC:omega}
\label{omega}
 For $1\leq i \leq n$, let $k_i$ be a non-negative integer, and let $K = \sum_{i=1}^n k_i$. 
For any partition $\calP  = \{P_1,\dots, P_r \}
\vdash [n]$, define $\alpha_j := \sum_{i\in P_j} k_i$.
With notation  as in the previous paragraph, the following formula holds in  $R^{K}(\Moduli_{g,n})$:
\begin{align}
\label{omegafor}
\prod_{i=1}^n\omega_i^{k_i} = \sum_{\calP \,\vdash [n]}[\Delta_{\calP}]\prod_{j=1}^{\ell(\calP)} \frac{\psi_{\bullet_j}^{\alpha_j}}{(-\psi_{\bullet_j}-\psi_{\star_j})^{1-\delta_j}},
\end{align}
where $\delta_j = \delta_{1,|P_j|}$ is a Kronecker delta and
we follow the  convention of considering negative powers of $\psi$  equal to $0$.
\end{proposition}

Specializing to monomials of top degree and using projection formula one can deduce a formula relating intersection numbers of $\kappa$ classes on $\Moduli_{g}$ with intersection numbers of $\psi$ classes. 

\begin{corollary}\cite[Cor. 7.10]{alexeev-guy}\cite[Thm 3.1, Cor. 3.5]{BC:omega}
\label{num}
For $1\leq i \leq n$, let $k_i$ be a non-negative integer, and let $ \sum_{i=1}^n k_i = 3g-3+n$. 
For any partition $\calP  = \{P_1,\dots, P_r \}
\vdash [n]$, define $\alpha_j := \sum_{i\in P_j} k_i$.
\begin{align}
\label{omegafor2}
\int_{\Moduli_{g}}\prod_{i=1}^n\kappa_{k_i-1}=\int_{\Moduli_{g,n}}\prod_{i=1}^n\omega_i^{k_i} = \sum_{\calP \,\vdash [n]}(-1)^{n+\ell(\calP)} \int_{\Moduli_{g,\ell(\calP)}}\prod_{i=1}^{\ell(\calP)} \psi_{\bullet_i}^{\alpha_i-|P_i|+1}.
\end{align}
\end{corollary}
\section{Generating functions and Differential Operators}
\label{sec:gf}

In this section we introduce generating functions for intersection numbers of $\psi, \omega$ and $\kappa$ classes, and we show that Corollary \ref{num} gives rise to functional equations relating these potentials. We first introduce the generating function for intersection numbers of $\psi$ classes, also known as the Gromov-Witten potential of a point.

Denote $\vec{\tau}:= \sum_{0}^{\infty} t_i \tau_i$, where the $t_i$'s are coordinates on a countably dimensional vector space $\bH^+$ with basis given by the vectors $\tau_i$. The Witten brackets, depending on $g,n$,  define  multilinear functions on $\bH^+$:
\begin{equation}
\langle \tau_0^{n_0}\cdots \tau_m^{n_m}\rangle_{g,n} = \int_{\overline{M}_{g,n}} \prod_{i= n_0+1}^{n_0+n_1}\psi_i \prod_{i= n_0+n_1+1}^{n_0+n_1+n_2}\psi_i^2 \cdot \ldots \cdot  \prod_{i= n_0+n_1+\ldots +n_{m-1}+1}^{n}\psi_i^m,
\end{equation}

 where $n = \sum_{j=0}^m n_j$
and
$
3g-3+n = \sum_{j=0}^m jn_j.
$

\begin{definition}
The {\bf genus $g$ Gromov-Witten potential of a point} is defined to be
\begin{align*}
\calF^g(t_0, t_1,\ldots) &=  \langle e^{\vec{\tau}}\rangle_g = \sum_{n=0}^{\infty} \frac{1}{n!} \langle {\vec{\tau}}, \ldots, {\vec{\tau}} \rangle_{g,n} .
\end{align*}

The {\bf total Gromov-Witten potential} is obtained by summing over all genera and adding a formal variable $u$ keeping track of genus (more precisely, of Euler characteristic):
\begin{align*}
\calF(u; t_0, t_1,\ldots) =  \sum_{g=0}^{\infty} u^{2g-2} \calF^g(t_0, t_1,\ldots) .
\end{align*}

\end{definition}

\begin{remark}
The Gromov-Witten potential is an exponential generating function for intersection numbers of $\psi$ classes. Because such intersection numbers are invariant under the action of the symmetric group permuting the marks, the variable $t_i$ does not refer to the insertion at the $i$-th mark, but rather to an insertion (at some mark) of $\psi^i$. For example, the intersection number $\int_{\Moduli_{2,6}} \psi_1^6\psi_2\psi_3\psi_4$ is the coefficient of the monomial $u^2\frac{t_0^2}{2!}\frac{t_1^3}{3!}t_6$ in $\mathcal{F}$.
\end{remark}

We define  similar generating functions for $\omega$ classes.
Denote $\vec{\sigma}:= \sum_{0}^{\infty} s_i \sigma_i$, and
\begin{equation}
\langle \sigma_0^{n_0}\cdots \sigma_m^{n_m}\rangle_{g,n} = \int_{\overline{M}_{g,n}} \prod_{i= n_0+1}^{n_0+n_1}\omega_i \prod_{i= n_0+n_1+1}^{n_0+n_1+n_2}\omega_i^2 \cdot \ldots \cdot  \prod_{i= n_0+n_1+\ldots +n_{m-1}+1}^{n}\omega_i^m.
\end{equation}

\begin{definition}
The {\bf genus $g$ $\omega$-potential of a point} is defined to be
\begin{align*}
\calS^g(s_0, s_1,\ldots) &=  \langle e^{\vec{\sigma}}\rangle_g = \sum_{n=0}^{\infty} \frac{1}{n!} \langle {\vec{\sigma}}, \ldots, {\vec{\sigma}} \rangle_{g,n} 
\end{align*}

The {\bf total $\omega$-potential}  is:
$$
\calS(u; s_0, s_1,\ldots) =  \sum_{g=1}^{\infty} u^{2g-2} \calS^g(s_0, s_1,\ldots) .
$$

\end{definition}

The analysis of push-forwards of $\omega$ classes yields some immediate results about the structure of $\calS$. 
\begin{lemma} \label{expo}
\begin{equation} \calS^g = e^{(2g-2)s_1}\tilde{\calS}^g(s_2, s_3, \ldots),
\end{equation}
where $\tilde{\calS}^g$ is some function depending only on variables $s_i$, with $i\geq 2$.
\end{lemma}
\begin{proof}
Lemma \ref{ostring} implies that $\calS$ is constant in $s_0$; the statement of Lemma \ref{odil}  is equivalent to showing that $S^g$ satisfies the differential equation 
\begin{equation}
\frac{\partial{\calS^g}}{\partial s_1} = (2g-2) \calS^g,
\end{equation}
and therefore it depends exponentially on $s_1$.
\end{proof}

Finally, we introduce a potential for intersection numbers of $\kappa$ classes on $\Moduli_{g}$.
\begin{definition}\label{def:kapot}
Let $p_0, p_1, \ldots$ be a countable set of formal variables, and $\mu = (\mu_i)$ denote a partition of an integer such that $w(\mu):=\sum i\mu_i = 3g-3$. Then the {\bf $\kappa$-potential} is:
\begin{align*}
\calK(p_0, p_1, \ldots) = \frac{p_0}{24} + \sum_{\substack{g\geq 2 \\ w(\mu)= 3g-3}} \left( \int_{\Moduli_{g}}\prod \kappa_i^{\mu_i} \right) \prod\frac{p_i^{\mu_i}}{\mu_i!}. 
\end{align*}
\end{definition}

\begin{remark} Some observations regarding this definition:
\begin{itemize}
\item the part of the potential corresponding to a fixed genus $g$ is  polynomial;
\item the genus variable $u$ is omitted, since for any monomial, the genus for the intersection number of the corresponding coefficient  is easily recovered from the weight $w(\mu)$.
\end{itemize}
\end{remark}

The potentials $\calK$ and $\calS$ are very closely related.
\begin{lemma}\label{l:KSS}\begin{equation}\label{KSS}
\calK(p_0, p_1, \ldots) = \calS(1; 0, p_0, p_1, \ldots) =  \calS(p_0; 0, 0, p_1, \ldots).
\end{equation}
\end{lemma}
\begin{proof}
Lemma \ref{oka} implies that $\calK$ is obtained from $\calS$ by setting $u=1$ and shifting variables so that $s_{i+1} = p_{i}$. Then the structure statement of Lemma \ref{expo} implies the second equality in \eqref{KSS}.
\end{proof}
The equalities \eqref{KSS} motivate the introduction of the ``unstable" term  $p_0/24$, which corresponds to assigning value $1/24$ to $\kappa_0$ on $\Moduli_{1}$ (which is not a Deligne-Mumford stack).

The combinatorial formulae of Corollary \ref{num} can be rephrased as the existence of a differential operator which acts on the Gromov-Witten potential to produce the $\omega$-potential.

\begin{theorem}\label{diffop}
Define the {\bf fork} operator as:
\begin{equation}\label{fork}
\mathcal{L} := \sum_{n=1}^\infty  \frac{(-1)^{n-1}}{n!}\sum_{i_1,\dots,i_n = 0}^\infty s_{i_1}\cdots s_{i_n} \vardel_{t_{i_1 + \cdots +i_n +1-n}} .
\end{equation}
Denote by $\tilde{\calF} = \sum_{g=1}^\infty u^{2g-2}\calF_g$ the positive genus part of the Gromov-Witten potential. Then:
\begin{equation} \label{expop}
(e^\calL \tilde{\calF})_{|\mathbf{t}=0} = \calS
\end{equation}
\end{theorem}

\begin{remark}
Before we embark in a formal proof of Theorem \ref{diffop}, we provide some intuition. Consider the summand in the generating function $\calS$: $(\int_{\Moduli_{2,5}}\omega_1^2\omega_2^2\omega_3^2 \omega_4\omega_5) u^2\frac{s_1^2}{2}\frac{s_2^3}{3!}$; consider the operator $s_1s_2 \partial_{t_2}$ (up to sign, a term of $\calL$). Let us assume that there is a  function $\calG$ such that $\calS = s_1s_2 \partial_{t_2} \calG$. What we want to understand is what kind of information we get about $\calG$ by equating the coefficients of the monomial  $\frac{s_1^2}{2}\frac{s_2^3}{3!}$ in the above differential equation. A simple calculus exercise shows that if $\calG = \ldots +\gamma u^2t_2s_1 \frac{s_2^2}{2}+\ldots$, then
\begin{equation}\label{grab}
\int_{\Moduli_{2,5}}\omega_1^2\omega_2^2\omega_3^2 \omega_4\omega_5 =6\gamma.
\end{equation}
Let us now give a diagrammatic representation of this relation. Associate a graph to a monomial in $s,t$ variables as in Figure \ref{fig-gra}. Let $\Gamma$ denote the graph associated to  the monomial $(\int_{\Moduli_{2,5}}\omega_1^2\omega_2^2\omega_3^2 \omega_4\omega_5) u^2\frac{s_1^2}{2}\frac{s_2^3}{3!}$. Then \eqref{grab} is represented in Figure \ref{fig-grab}: applying the operator $s_1s_2 \partial_{t_2}$ compares $\Gamma$ with  the  sum of all graphs obtained by grabbing a leg of $\Gamma$ decorated by $\omega_i^1$, a leg decorated by $\omega_j^2$ and replacing them with a leg decorated with a $\psi^2$.
In general, an operator of the form $s^I\partial_{t_k}$ compares the coefficient of $\calS$ corresponding to a graph $\Gamma$, to coefficients of $\calG$ obtained by grabbing in all possible ways a collection of $\omega$-legs  of $\Gamma$ decorated according to the multi-index $I$, and replacing them by a unique leg decorated by $\psi^k$. The operator $\calL$ runs over all possibilities of grabbing a some  $\omega$-legs and replacing them with a unique leg; this leg is decorated with a power of $\psi$ determined by the number and the powers of $\omega$ classes that have been grabbed. Applying the exponential of the operator $\calL$ then corresponds to grabbing multiple  groups of $\omega$ legs and replacing each group with a $\psi$ leg, not caring about the order of the grabbing. At this point it should be possible to recognize, when applying $e^\calL$ to the Gromov-Witten potential $\tilde\calF$, the equations \eqref{omegafor2}.  The operator $\calL$ is constructed precisely to reproduce the formula based on the summation over all partitions of the index set in \eqref{omegafor2}.


\end{remark}

 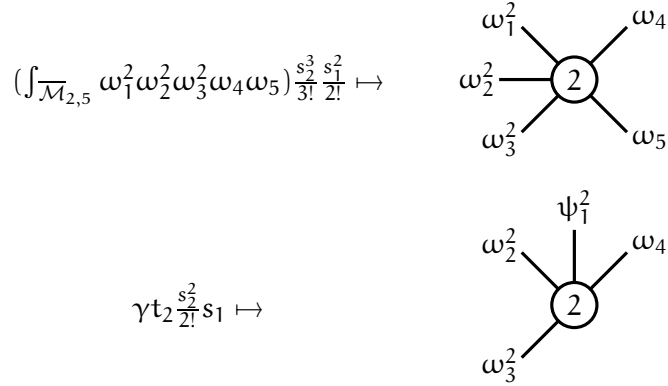
\begin{figure}[tb]
\begin{tikzpicture}

\node at (-5,0) {$(\int_{\Moduli_{2,5}}\omega_1^2\omega_2^2\omega_3^2 \omega_4\omega_5) \frac{s_2^3}{3!}\frac{s_1^2}{2!}\mapsto$};
\draw[very thick]  (0,0) circle (0.30);
\node at (0,0) {$2$};
\node at (1,0.8) {$\omega_4$};
\node at (1,-0.8) {$\omega_5$};
\node at (-1,0.8) {$\omega_1^2$};
\node at (-1,-0.8) {$\omega_3^2$};
\node at (-1.3,0) {$\omega_2^2$};
\draw[very thick]  (-1,0) -- (-0.30,0);
\draw[very thick]  (-0.7,0.7) -- (-0.20,0.2);
\draw[very thick]  (-0.7,-0.7) -- (-0.20,-0.2);
\draw[very thick]  (0.7,0.7) -- (0.20,0.2);
\draw[very thick]  (0.7,-0.7) -- (0.20,-0.2);
\begin{scope}[shift ={(0,-3)}]
\node at (-5,0) {$\gamma  t_2\frac{s_2^2}{2!}s_1\mapsto$};
\draw[very thick]  (0,0) circle (0.30);
\node at (0,0) {$2$};
\node at (1,0.8) {$\omega_4$};
\node at (-1,0.8) {$\omega_2^2$};
\node at (-1,-0.8) {$\omega_3^2$};
\node at (0,1.3) {$\psi_1^2$};
\draw[very thick]  (0,1) -- (0, 0.30);
\draw[very thick]  (-0.7,0.7) -- (-0.20,0.2);
\draw[very thick]  (-0.7,-0.7) -- (-0.20,-0.2);
\draw[very thick]  (0.7,0.7) -- (0.20,0.2);
\end{scope}

\end{tikzpicture}
\caption{Decorated graphs are associated to monomials in $s$ and $t$ variables. Such graphs also keep track of the coefficients of the monomial being an intersection number of $\psi$ and $\omega$ classes.}
\label{fig-gra}
\end{figure} 

 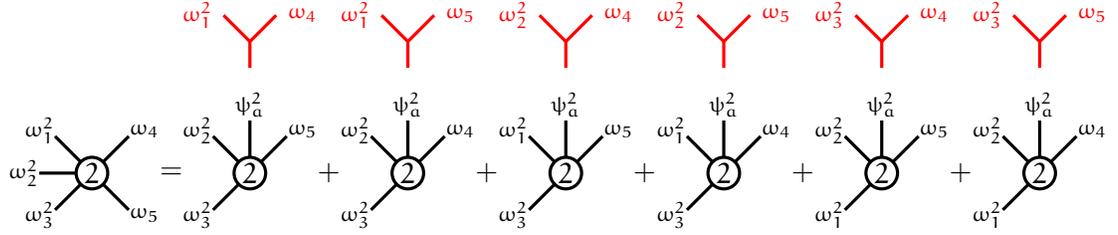
\begin{figure}[tb]
\begin{tikzpicture}

\begin{scope}[scale= 0.7]
\draw[very thick]  (0,0) circle (0.30);
\node at (0,0) {$2$};
\node at (1.5,0) {$=$};
\node at (1,0.8) {\scriptsize$\omega_4$};
\node at (1,-0.8) {\scriptsize$\omega_5$};
\node at (-1,0.8) {\scriptsize$\omega_1^2$};
\node at (-1,-0.8) {\scriptsize$\omega_3^2$};
\node at (-1.3,0) {\scriptsize$\omega_2^2$};
\draw[very thick]  (-1,0) -- (-0.30,0);
\draw[very thick]  (-0.7,0.7) -- (-0.20,0.2);
\draw[very thick]  (-0.7,-0.7) -- (-0.20,-0.2);
\draw[very thick]  (0.7,0.7) -- (0.20,0.2);
\draw[very thick]  (0.7,-0.7) -- (0.20,-0.2);

\begin{scope}[shift ={(3,0)}]
\draw[very thick]  (0,0) circle (0.30);
\node at (0,0) {$2$};
\node at (1,0.8) {\scriptsize$\omega_5$};
\node at (-1,0.8) {\scriptsize$\omega_2^2$};
\node at (-1,-0.8) {\scriptsize$\omega_3^2$};
\node at (0,1.3) {\scriptsize $\psi_a^2$};
\draw[very thick]  (0,1) -- (0, 0.30);
\draw[very thick]  (-0.7,0.7) -- (-0.20,0.2);
\draw[very thick]  (-0.7,-0.7) -- (-0.20,-0.2);
\draw[very thick]  (0.7,0.7) -- (0.20,0.2);
\draw[very thick, color = red]  (0,2) -- (0, 2.5) -- (0.5,3) ;
\draw[very thick, color = red]   (0, 2.5) -- (-0.5,3) ;
\node[color = red] at (-1,3) {\scriptsize $\omega_1^2$};
\node[color = red] at (1,3) {\scriptsize $\omega_4$};
\node at (1.5,0) {$+$};
\end{scope}

\begin{scope}[shift ={(6,0)}]
\draw[very thick]  (0,0) circle (0.30);
\node at (0,0) {$2$};
\node at (1,0.8) {\scriptsize$\omega_4$};
\node at (-1,0.8) {\scriptsize$\omega_2^2$};
\node at (-1,-0.8) {\scriptsize$\omega_3^2$};
\node at (0,1.3) {\scriptsize $\psi_a^2$};
\draw[very thick]  (0,1) -- (0, 0.30);
\draw[very thick]  (-0.7,0.7) -- (-0.20,0.2);
\draw[very thick]  (-0.7,-0.7) -- (-0.20,-0.2);
\draw[very thick]  (0.7,0.7) -- (0.20,0.2);
\draw[very thick, color = red]  (0,2) -- (0, 2.5) -- (0.5,3) ;
\draw[very thick, color = red]   (0, 2.5) -- (-0.5,3) ;
\node[color = red] at (-1,3) {\scriptsize $\omega_1^2$};
\node[color = red] at (1,3) {\scriptsize $\omega_5$};
\node at (1.5,0) {$+$};
\end{scope}


\begin{scope}[shift ={(9,0)}]
\draw[very thick]  (0,0) circle (0.30);
\node at (0,0) {$2$};
\node at (1,0.8) {\scriptsize$\omega_5$};
\node at (-1,0.8) {\scriptsize$\omega_1^2$};
\node at (-1,-0.8) {\scriptsize$\omega_3^2$};
\node at (0,1.3) {\scriptsize $\psi_a^2$};
\draw[very thick]  (0,1) -- (0, 0.30);
\draw[very thick]  (-0.7,0.7) -- (-0.20,0.2);
\draw[very thick]  (-0.7,-0.7) -- (-0.20,-0.2);
\draw[very thick]  (0.7,0.7) -- (0.20,0.2);
\draw[very thick, color = red]  (0,2) -- (0, 2.5) -- (0.5,3) ;
\draw[very thick, color = red]   (0, 2.5) -- (-0.5,3) ;
\node[color = red] at (-1,3) {\scriptsize $\omega_2^2$};
\node[color = red] at (1,3) {\scriptsize $\omega_4$};
\node at (1.5,0) {$+$};
\end{scope}


\begin{scope}[shift ={(12,0)}]
\draw[very thick]  (0,0) circle (0.30);
\node at (0,0) {$2$};
\node at (1,0.8) {\scriptsize$\omega_4$};
\node at (-1,0.8) {\scriptsize$\omega_1^2$};
\node at (-1,-0.8) {\scriptsize$\omega_3^2$};
\node at (0,1.3) {\scriptsize $\psi_a^2$};
\draw[very thick]  (0,1) -- (0, 0.30);
\draw[very thick]  (-0.7,0.7) -- (-0.20,0.2);
\draw[very thick]  (-0.7,-0.7) -- (-0.20,-0.2);
\draw[very thick]  (0.7,0.7) -- (0.20,0.2);
\draw[very thick, color = red]  (0,2) -- (0, 2.5) -- (0.5,3) ;
\draw[very thick, color = red]   (0, 2.5) -- (-0.5,3) ;
\node[color = red] at (-1,3) {\scriptsize $\omega_2^2$};
\node[color = red] at (1,3) {\scriptsize $\omega_5$};
\node at (1.5,0) {$+$};
\end{scope}

\begin{scope}[shift ={(15,0)}]
\draw[very thick]  (0,0) circle (0.30);
\node at (0,0) {$2$};
\node at (1,0.8) {\scriptsize$\omega_5$};
\node at (-1,0.8) {\scriptsize$\omega_2^2$};
\node at (-1,-0.8) {\scriptsize$\omega_1^2$};
\node at (0,1.3) {\scriptsize $\psi_a^2$};
\draw[very thick]  (0,1) -- (0, 0.30);
\draw[very thick]  (-0.7,0.7) -- (-0.20,0.2);
\draw[very thick]  (-0.7,-0.7) -- (-0.20,-0.2);
\draw[very thick]  (0.7,0.7) -- (0.20,0.2);
\draw[very thick, color = red]  (0,2) -- (0, 2.5) -- (0.5,3) ;
\draw[very thick, color = red]   (0, 2.5) -- (-0.5,3) ;
\node[color = red] at (-1,3) {\scriptsize $\omega_3^2$};
\node[color = red] at (1,3) {\scriptsize $\omega_4$};
\node at (1.5,0) {$+$};
\end{scope}


\begin{scope}[shift ={(18,0)}]
\draw[very thick]  (0,0) circle (0.30);
\node at (0,0) {$2$};
\node at (1,0.8) {\scriptsize$\omega_4$};
\node at (-1,0.8) {\scriptsize$\omega_2^2$};
\node at (-1,-0.8) {\scriptsize$\omega_1^2$};
\node at (0,1.3) {\scriptsize $\psi_a^2$};
\draw[very thick]  (0,1) -- (0, 0.30);
\draw[very thick]  (-0.7,0.7) -- (-0.20,0.2);
\draw[very thick]  (-0.7,-0.7) -- (-0.20,-0.2);
\draw[very thick]  (0.7,0.7) -- (0.20,0.2);
\draw[very thick, color = red]  (0,2) -- (0, 2.5) -- (0.5,3) ;
\draw[very thick, color = red]   (0, 2.5) -- (-0.5,3) ;
\node[color = red] at (-1,3) {\scriptsize $\omega_3^2$};
\node[color = red] at (1,3) {\scriptsize $\omega_5$};

\end{scope}

\end{scope}
\end{tikzpicture}
\caption{The comparison of one of the coefficients of $\calS$ and $\calG$ given by the differential equation $\calS = s_1s_2 \partial_{t_2} \calG$.}
\label{fig-grab}
\end{figure} 

\begin{proof}
A formal proof of Theorem \ref{diffop} is an exercise in bookkeeping: for any given monomial, we show that the coefficients on either side of \eqref{expop} agree, by using formula \eqref{omegafor2}.

Fix a monomial $u^{2g-2}\frac{s_1^{k_1}}{k_1!}\ldots \frac{s_n^{k_n}}{k_n!}$; we let  $K = 1^{k_1}, 2^{k_2}, \ldots, m^{k_m}$ denote a multi-index where the first $k_1$ indices have value $1$ , and so on. We do not care about the order of the values, so all multi-indices arising in this proof can be assumed normalized so that the values are non-decreasing. The coefficient for this monomial on the right hand side of \eqref{expop} is
$\int_{\Moduli_{g,n}}\omega^K$, where we use multi-index notation in the natural way. Formula \eqref{omegafor2} gives an expression for this quantity in terms of a weighted sum over all partitions of the index set, so let us fix a partition $\calP = P_1, \ldots , P_r$ of the index set. Each part $P_i$ of the partition $\calP$ produces a multi-index by looking at the powers of $\omega$ classes supported on the points that belong to $P_i$. It is possible that different parts give rise to the same multi-index. We denote $\underline{J} = J_1^{v_1}, \ldots J_t^{v_t}$ the collection of the multi-indices arising from the parts of $\calP$, intending that the multi-index $J_1$ arises $v_1$ times, and so on. For  $i$ from $1$ to $t$, we denote  $J_i = 1^{j_{i,1}}, \ldots, m^{j_{i,m}}$.
Finally, from $\underline{J}$, we can produce a multi-index $\alpha = \alpha_1^{v_1}, \ldots,\alpha_t^{v_t}$, where $\alpha_i = 1 +  j_{i,2}+2j_{i,3}+\ldots+(m-1) j_{i,m}$ (note that this is one plus the sum of the values minus the number of the parts of $J_i$, as in the definition of the exponents of the $\psi$ classes in \eqref{omegafor2} ). The multi-index $\alpha$ is the exponent vector of the monomial in $\psi$ classes corresponding to the partition $\calP$ in \eqref{omegafor2}. It is possible that different partitions of the set of indices give rise to the same multi-index $\alpha$: in fact there are exactly $\prod k_i!/((\prod j_{i,l}!)^{v_i}\prod v_i!)$ distinct partitions of the indices that will produce $\alpha$. We can rewrite \eqref{omegafor2} as a summation over the combinatorial data given by $\underline{J}$:
\begin{equation}\label{whew}
\int_{\Moduli_{g,n}}\omega^K = \sum_{\underline{J}} (-1)^{n+\ell(\alpha)} \frac{\prod k_i!}{(\prod j_{i,l}!)^{v_i}\prod v_i!}\int_{\Moduli_{g,\ell(\alpha)}}\psi^\alpha
\end{equation}
We exhibit a summand $m_\alpha$ in $\tilde\calF$, and a term $L_{\underline{J}}$ in the differential operator $e^\calL$ such that 
$L_{\underline{J}}m_\alpha$ is a multiple of $u^{2g-2}\frac{s^{K}}{K!}$;  such multiple may be shown to be $\prod k_i!/((\prod j_{i,l}!)^{v_i}\prod v_i!)$.
Define:
\begin{align}
m_\alpha & =  \left(\int_{\Moduli_{g,\ell(\alpha)}}\psi^\alpha\right) u^{2g-2} \frac{t_{\alpha_1}^{v_1}}{v_1!}\ldots \frac{t_{\alpha_t}^{v_t}}{v_t!}\\
L_{\underline{J}} & =\frac{ (-1)^{n+\ell(\alpha)} }{\prod v_i!} \prod \left( \frac{s_1^{j_{i,1}}}{j_{i,1}!}\ldots \frac{s_m^{j_{i,m}}}{j_{i,m}!}\partial_{t_{\alpha_i}}\right)^{v_i}
\end{align}
One may show that all pairs of terms  $m\in \tilde\calF$, $L\in e^\calL$ such that $Lm$ is a multiple  of $u^{2g-2}\frac{s^{K}}{K!}$ arise in this fashion. It follows that the coefficient of $u^{2g-2}\frac{s^{K}}{K!}$ in $e^\calL \tilde\calF$ equals the right hand side of \eqref{whew}, and therefore it agrees with the coefficient of the same monomial in  $\calS$. This concludes the proof of Theorem \ref{diffop}.
\end{proof}

\section{Change of Variables}
\label{sec:cov}


Theorem \ref{diffop} states that $\calS$ is the restriction of a function obtained by applying the exponential\footnote{it is important to note that the coefficients of the vector field are functions in variables that commute with the $\vardel_{t_i}$'s.} of a vector field to the function $\tilde\calF$. It follows that the two generating functions are related by a change of variables. In this section, we derive explicitly this change of variables, and then deduce an equivalent statement for the potential for $ \kappa$ classes $\calK$.

\begin{lemma} \label{expflow}
With notation as above, and denoting $\DD\mathcal{L} = \sum_{i = 0}f_i(\bm{s})\vardel_{t_i}$,
\begin{align}
e^\mathcal{L}\mathcal{F}(\bm{t}) &= \mathcal{F}(\bm{t} + \bm{f}(\bm{s})). 
\end{align}
\end{lemma}

\begin{proof}
This is just the exponential flow in differential geometry.  For the benefit of readers who may not be familiar with this technique, we provide the sketch of an elementary proof.
In one variable it is a Taylor expansion argument: imagine that $\calL = f(s)\vardel_t$ and $\tilde\calF$ is a function of one variable $t$, then
\begin{align*}
e^\mathcal{L}\mathcal{F}(t) &= \left(1 + f(s)\vardel_{t} + \frac{1}{2}f^2(s)\vardel^2_{t}+\cdots\right)\mathcal{F}(t) \\
&= \sum_{n=0}^\infty \frac{\mathcal{F}^{(n)}(t)}{n!}\left(f(s)\right)^n \\
&= \mathcal{F}(t+f(s)) \text{ expanded at } t.
\end{align*}
In countably many variables one has:
\begin{align*}
e^\mathcal{L}\mathcal{F}(\bm{t}) &= e^{\sum f_i(\bm{s})\vardel_{t_i}}\mathcal{F}(\bm{t}) \\
&=\prod e^{f_i(\bm{s})\vardel_{t_i}} \mathcal{F}(\bm{t})\\
&=  \mathcal{F}(\bm{t} + \bm{f}(\bm{s})), 
\end{align*}
where the variables have been translated one at a time.
\end{proof}

\begin{corollary}
The potential $\calS$ is obtained from $\tilde\calF$ via the change of variables encoded in the following generating function:
\begin{equation}\label{covgf}
\sum_{i=0}^\infty t_i z^i = \left[z\left(1- e^{\sum_{k=0}^\infty -s_k z^{k-1}}\right)\right]_+,
\end{equation}
where the subscript $+$ denotes the truncation of the expression in parenthesis to terms with non-negative exponents for the variable $z$.
\end{corollary}
\begin{proof}
Combining Theorem \ref{diffop} and Lemma \ref{expflow}, we obtain that $\calS(u;\mathbf{s} ) = \tilde\calF(u;f_0(\mathbf{s}), f_1(\mathbf{s}), \ldots)$, where $f_i$ is the coefficient of $\vardel_{t_i}$ in the vector field $\calL$. Resumming the coefficients to express $\calL$ in the form $\sum_{i = 0}f_i(\bm{s})\vardel_{t_i}$  one obtains:
\begin{align*}
t_0 = f_0(\mathbf{s}) &= s_0 -s_0s_1 +\frac{s_0^2}{2!}s_2 +s_0\frac{s_1^2}{2!} -\frac{s_0^3}{3!}s_3 - \frac{s_0^2}{2!}s_1 s_2 - s_0\frac{s_1^3}{3!} + \ldots \\
t_1 = f_1(\mathbf{s}) &= s_1  -{s_0}s_2 - \frac{s_1^2}{2!} +\frac{s_0^2}{2!}s_3 +{s_0}s_1 s_2 + \frac{s_1^3}{3!} - \frac{s_0^3}{3!}s_4- \ldots \\
t_2 = f_2(\mathbf{s}) &= s_2 - {s_0}s_3 -s_1 s_2  + \frac{s_0^2}{2!}s_4+ s_0 s_1s_3 + s_0\frac{s_2^2}{2!}+\frac{s_1^2}{2!}s_2-  \ldots \\
t_3 = f_3(\mathbf{s}) &= s_3   - {s_0}s_4-  s_1s_3 - \frac{s_2^2}{2!} + \frac{s_0^2}{2!}s_5 + s_0s_1s_4+ s_0 s_2s_3+ \frac{s_1^2}{2!}s_3+s_1 \frac{s_2^2}{2!}-  \ldots \\
\end{align*}
It is a simple combinatorial exercise to see that the above change of variables is organized in generating function form as in \eqref{covgf}.
\end{proof}

We have at this point all the technical information necessary to give our proof of \cite[Thm 4.1]{mz}. Before doing so, we establish one final piece of notation.

\begin{definition}\label{def:sch}
For $i\geq 0$ and a countable set of variables $\mathbf{p} = p_1, p_2, \ldots$, we define the functions $S_i(\mathbf{p})$ via:
\begin{equation}
\sum_{i=0}^\infty S_i(\mathbf{p}) z^i = e^{\sum_{k=1}^\infty {p_k}z^k}
\end{equation}
\end{definition}

\begin{remark}
The functions $S_i$ appear in  some literature (especially in the school of integrable systems) as {\it elementary Schur polynomials}. The reason is that  by interpreting the variables $p_k$ as (normalized) power sums,
$$
p_k = \frac{\sum x_j}{j},
$$ 
then $S_i$ is the {\it complete symmetric polynomial of degree $k$} in the variables $x_j$, which also is the Schur polynomial associated to the one part partition $k$. 
\end{remark}

\begin{theorem} \label{thm:mz}
The generating functions $\calK(p_0,\mathbf{p})$ and $\tilde\calF(u;t_0, t_1, \ldots)$ agree when restricting the domain of $\tilde\calF$ to the subspace $t_0= t_1=0$ and applying the following transformation to the remaining variables:
\begin{align*}
u &= e^{p_0},\\
t_i & = -S_{i-1}(-\mathbf{p}).
\end{align*}
\end{theorem}

\begin{proof}
By the $\omega$-string relation (Lemma \ref{ostring}), the function $\calS$ is constant in $s_0$, therefore, one  may restrict the domain to the hyperplane $s_0 =0$. Imposing this restriction on the transformation given in \eqref{covgf}, one deduces that the domain of the function $\tilde\calF$ is restricted to the hyperplane $t_0=0$. After reindexing, one  obtains: 

\begin{align}\label{covgfreg}
\sum_{i=0}^\infty t_{i+1} z^i &= 1- e^{\sum_{k=0}^\infty -s_{k+1} z^{k}}.
\end{align}
For $k \geq 0$, define $s_{k+1} = p_k$, to obtain the change of variables:
\begin{align}
t_1 &= 1- e^{-p_0} \nonumber \\
t_i &=  -e^{-p_0}S_{i-1}(-p_1,- p_2, \ldots), \mbox{for $i\geq 2$.} \label{regcov}
\end{align}

Combining the statement of Corollary \ref{covgf} with  Lemma \ref{l:KSS}, one obtains
\begin{align*}
\calK(p_0, p_1, \ldots)&= \calS(1;0, p_0, p_1, \ldots)  = \calS(e^{p_0};0, 0, p_1, \ldots)\\
&= \tilde\calF(1;0, 1- e^{-p_0},  -e^{-p_0}S_{1}(-p_1, -p_2, \ldots),  -e^{-p_0}S_{2}(-p_1, -p_2, \ldots), \ldots)\\
& = \tilde\calF(e^{p_0};0, 0, S_{1}(-p_1, -p_2, \ldots), S_{2}(-p_1, -p_2, \ldots), \ldots).
\end{align*}
\end{proof}

\section{Virasoro Relations for $\kappa$ Potential}
\label{sec:chain}


The goal of this Section is to obtain a countable number of differential equations that annihilate the potential $\calK$, and determine recursively all intersection numbers of $\kappa$ classes on $\Moduli_g$ from the  unstable term $\frac{1}{24}p_0$. We start from the Virasoro constrains annihilating and determining the Gromov-Witten potential $\calF$.

\begin{theorem}[Witten's conjecture, Kontsevich's theorem]
Consider the differential operators $L_n$ defined as follows:
\begin{align}
L_{-1} =& - \vardel_{t_0} +  \sum_{i=0}^\infty t_{i+1}{\vardel_{t_{i}}}+\frac{t_0^2}{2}, \label{eq:str}\\
L_{0} = & -\frac{3}{2} \vardel_{t_1}+   \sum_{i=0}^\infty \frac{(2i+1)}{2}t_{i}{\vardel_{t_{i}}}+\frac{1}{16}, \label{eq:dil}
\end{align}
and for all positive $n$,
\begin{align}
L_n = &-\left(\frac{(2n+3)!!}{2^{n+1}}\right){\vardel_{t_{n+1}}} + \sum_{i=0}^\infty \left(\frac{(2i+2n+1)!!}{(2i-1)!!2^{n+1}}\right)t_i{\vardel_{t_{i+n}}} + \nonumber \\ & \frac{u^2}{2}\sum_{i=0}^{n-1}\left(\frac{(2i+1)!!(2n-2i-1)!!}{2^{n+1}}\right){\vardel_{t_i} \vardel_{t_{n-1-i}}}. \label{eq:higher}
\end{align} 
For all $n\geq -1$, we have
\begin{equation}\label{vira}
L_n (e^{\calF}) = 0.
\end{equation}
Further, the equations  \eqref{vira} determine recursions among the coefficients of $\calF$ that recover uniquely the Gromov-Witten potential from the initial condition $\calF = \frac{t_0^3}{3!}+ higher \ order \ terms$.
\end{theorem}
Equation \eqref{eq:str} is equivalent to the string recursion (Lemma \ref{str}), equation \eqref{eq:dil} to  dilaton (Lemma \ref{dil}). For positive values of $n$, there is not a geometric interpretation for the recursions on the coefficients of $\calF$ given by \eqref{eq:higher}. 

Theorem \ref{thm:mz} gives an explicit change of variables that equates the potential $\calK$ with the restriction of $\tilde\calF$ to the linear subspace $t_0=0, u=1$. Essentially, since $\calF$ is annihilated by the operators $L_n$, we  produce operators annihilating $\calK$ by rewriting the $L_n$'s in the variables $p_i$. There are a couple subtle points to address:
\begin{enumerate}
\item The function $\calF = \tilde\calF +\frac{\calF_0}{u^2}$, where $\calF_0$ is a multiple of $t_0^3$. Therefore:
$$\left(L_n e^{ \tilde\calF}\right)_{|t_0 = 0} = \left(L_ne^{ \calF} \right)_{|t_0 = 0},
$$ 
since for any $n$, $L_n$ has at most a quadratic term in $\vardel_{t_0}$.
\item The operations of applying the differential operator $L_n$ and restricting to the hyperplane $t_0=0$ do not commute. In Section \ref{remove}, we use the string equation $L_{-1}$ to replace the operator $\vardel_{t_0}$ with a differential operator in the remaining variables.
\end{enumerate}
The computation is long and technical; we break it down into several subsections to give the reader the chance to isolate segments of it that may be of interest.

\subsection{Auxiliary variables and notation}
In this section we make some elementary changes of variables and introduce some notation, in order to simplify the computation.
Let:
\begin{equation}
 \hat{t}_0 = -(t_1 -1), \ \ \ \ \ \hat{t}_i = -t_{i+1} \ \ \mbox{for $i=-1$  and $i\geq 1$,}  \ \ \ \ \ \hat{p}_k = -p_k \ \ \mbox{for all $k\geq 0$.}
\end{equation}

The translation of the variable $t_1$ by $1$ is the  well-known {\bf dilaton shift} in Gromov-Witten theory, which has the effect of making the operators $L_n$ homogeneous quadratic. The remaining reindexing and signs are chosen  to simplify the change of variables \eqref{regcov} to:
\begin{equation}\label{schurcov}
\sum_{i=0}^\infty \hat{t}_{i} z^i = e^{\sum_{k=0}^\infty \hat{p}_{k} z^{k}} ,
\end{equation}
or equivalently
\begin{equation}
\hat{t}_{i}= e^{\hat{p}_{0}}S_i(\hat{p}_1, \hat{p}_2, \ldots).
\end{equation}
In the hatted variables, the operators $L_n$ become:
\begin{align}
L_{-1} =&  \sum_{i=-1}^\infty \hat{t}_{i+1}{\vardel_{\hat{t}_{i}}}+\frac{\hat{t}_{-1}^2}{2}, \\
L_{0} = &   \sum_{i=-1}^\infty \frac{(2i+3)}{2}\hat{t}_{i}{\vardel_{\hat{t}_{i}}}+\frac{1}{16}, \\
L_n = & \sum_{i=-1}^\infty \left(\frac{(2i+2n+3)!!}{(2i+1)!!2^{n+1}}\right)\hat{t}_i{\vardel_{\hat{t}_{i+n}}} + \nonumber \\ & \frac{u^2}{2}\sum_{i=-1}^{n-2}\left(\frac{(2i+3)!!(2n-2i-3)!!}{2^{n+1}}\right){\vardel_{\hat{t}_i} \vardel_{\hat{t}_{n-3-i}}}.
\end{align}

\subsection{Removing $\hat{t}_{-1} = -t_0$}
\label{remove}
In this section we derive operators $\hat{L}_n$, for all $n\geq 0$ that annihilate the restriction  ${e^{\calF}}_{|\hat{t}_{-1} = 0}$. 
From the operator $L_{-1}$ we have:
\begin{equation}
\vardel_{\hat{t}_{-1}}e^{\calF} =  -\left(\sum_{i=0}^\infty \frac{\hat{t}_{i+1}}{\hat{t}_{0}} {\vardel_{\hat{t}_{i}}}+\frac{\hat{t}_{-1}^2}{2\hat{t}_{0}}\right)e^{\calF}. \label{stringaway}
\end{equation}
We may write the operators $L_n$ making sure that any monomial containing $\vardel_{\hat{t}_{-1}}$ has it as the righmost term. Then replacing $\vardel_{\hat{t}_{-1}}$ with the right hand side of \eqref{stringaway} produces an operator $\tilde{L}_n$ still annhiliating $e^{\calF}$ and not containing any derivative in $\hat{t}_{-1}$. The application of $\tilde{L}_n$ commutes with restriction to the hyperplane $\hat{t}_{-1}=0$, so we may define $\hat{L}_n = \left(\tilde{L}_n\right)_{|\hat{t}_{-1}=0}$. Up to some tedious but straightforward computation we have proved the following.
\begin{lemma} \label{lem:ann}
For all $n\geq 0$, the operators $\hat{L}_n$ defined below annhiliate $(e^\calF)_{|\hat{t}_{-1}=0}$:
\begin{align}
\hat{L}_{0} = &   \sum_{i=0}^\infty \frac{(2i+3)}{2}\hat{t}_{i}{\vardel_{\hat{t}_{i}}}+\frac{1}{16}, \\
\hat{L}_{1} = &   \sum_{i=0}^\infty \frac{(2i+3)(2i+5)}{4}\hat{t}_{i}{\vardel_{\hat{t}_{i+1}}}+ \nonumber\\
&+ \frac{u^2}{8\hat{t}_0^2}\left( \sum_{i=0}^\infty \left(\hat{t}_{i+2}-\frac{\hat{t}_{1}\hat{t}_{i+1}}{\hat{t}_{0}}\right)\vardel_{\hat{t}_{i}}+\sum_{i,j=0}^\infty \hat{t}_{i+1}\hat{t}_{j+1}\vardel_{\hat{t}_{i}}\vardel_{\hat{t}_{j}} \right),\\
\hat{L}_{2} = &   \sum_{i=0}^\infty \frac{(2i+3)(2i+5)(2i+7)}{8}\hat{t}_{i}{\vardel_{\hat{t}_{i+2}}}+ \nonumber\\
&+ \frac{3u^2}{8\hat{t}_0^2}\left( \sum_{i=0}^\infty \left(\hat{t}_{i+1}(1-{\hat{t}_{0}}\vardel_{\hat{t}_{0}})\right)\vardel_{\hat{t}_{i}} \right),\\
\hat{L}_n = & \sum_{i=0}^\infty \left(\frac{(2i+2n+3)!!}{(2i+1)!!2^{n+1}}\right)\hat{t}_i{\vardel_{\hat{t}_{i+n}}} + \nonumber \\ & \frac{u^2}{2}\sum_{i=0}^{n-3}\left(\frac{(2i+3)!!(2n-2i-3)!!}{2^{n+1}}\right){\vardel_{\hat{t}_i} \vardel_{\hat{t}_{n-3-i}}}+\nonumber \\ 
&-\frac{u^2}{\hat{t}_0}\frac{(2n-1)!!}{2^{n+1}}  \left( \vardel_{\hat{t}_{n-3}}+ \sum_{i=0}^\infty \hat{t}_{i+1}  \vardel_{\hat{t}_{i}}\vardel_{\hat{t}_{n-2}}\right).
\end{align}
\end{lemma}
The need to treat the first three cases separately comes from the fact that the variable $\hat{t}_0$ and the remaining variables appear asymmetrically in the right hand side of \eqref{stringaway}.
We now rewrite these operators in the variables $\hat{p}_k$. We separate the computation in two parts, corresponding to the coefficients of $u^0$ and $u^2$ in the operators. We start by carrying out explicitly the change of variable for the first operator.

\subsection{Warm-up: $\hat{L}_0$}
\label{covdil}

We first focus on the term 
\begin{equation}\label{firstcase}
\sum_{i=0}^\infty \hat{t}_{i}{\vardel_{\hat{t}_{i}}} = \sum_{i=0}^\infty \hat{t}_{i}\sum_{j=0}^\infty \frac{\vardel{\hat{p}_{m}}}{\vardel{\hat{t}_{i}}}\vardel_{\hat{p}_{m}} .
\end{equation}
Consider the change of variables \eqref{schurcov}. By taking a partial derivative with respect to a variable $\hat{t}_i$, we obtain:
\begin{equation}\label{parder}
\sum_{j=0}^\infty \frac{\vardel \hat{p}_j}{\vardel \hat{t}_{i}} z^{j-i} = e^{-\sum_{k=0}^\infty \hat{p}_{k} z^{k}}. 
\end{equation}
We observe that \eqref{parder} holds for any fixed value of $i$, which implies that the partial derivatives $ \frac{\vardel \hat{p}_j}{\vardel \hat{t}_{i}}$ depend only on the difference $j-i$. In other words, the Jacobian of the change of variables may be thought as an infinite upper triangular matrix which is constant along translates of the diagonal.
It follows that the coefficient of $\vardel_{\hat{p}_{m}} $ in \eqref{firstcase} is the coefficient of $z^m$ in the product of generating series:
\begin{align*}\label{genser}
\sum_{i=0}^\infty \hat{t}_{i}{\vardel_{\hat{t}_{i}}} &=\sum_{m=0}^\infty\left[\left(\sum_{i=0}^\infty \hat{t}_{i}z^i\right)\left(\sum_{j=0}^\infty \frac{\vardel{\hat{p}_{j}}}{\vardel{\hat{t}_{i}}}z^{j-i}\right)\right]_{z^m}\vardel_{\hat{p}_m} \\
& = \sum_{m=0}^\infty \left[e^{\sum_{k=0}^\infty \hat{p}_{k} z^{k}}e^{-\sum_{k=0}^\infty \hat{p}_{k} z^{k}}\right]_{z^m}\vardel_{\hat{p}_m} = \vardel_{\hat{p}_0}.
\end{align*}

We return to the change of variables \eqref{schurcov} and apply the operator $z\vardel_z$ to obtain
\begin{equation}
\sum_{i=0}^\infty i \hat{t}_{i} z^i = e^{\sum_{k=0}^\infty \hat{p}_{k} z^{k}} \sum_{k=0}^\infty k\hat{p}_{k} z^{k}
\end{equation}
With an analogous argument to the previous case, we obtain:
\begin{align*}
\sum_{i=0}^\infty i \hat{t}_{i}{\vardel_{\hat{t}_{i}}} 
& = \sum_{m=0}^\infty \left[\left(e^{\sum_{k=0}^\infty \hat{p}_{k} z^{k}} \sum_{k=0}^\infty k\hat{p}_{k} z^{k}\right)(e^{-\sum_{k=0}^\infty \hat{p}_{k} z^{k}})\right]_{z^m}\vardel_{\hat{p}_m} \\
&=  \sum_{m=0}^\infty  m\hat{p}_{m} \vardel_{\hat{p}_m}.
\end{align*}
We may now express the operator $\hat{L}_0$ in the variables $p_i = -\hat{p}_i$ as 
\begin{equation}\label{dilop}
\hat{L}_{0} =    \sum_{i=0}^\infty \frac{(2i+3)}{2}\hat{t}_{i}{\vardel_{\hat{t}_{i}}}+\frac{1}{16} = -\frac{3}{2} \vardel_{{p}_0}+  \sum_{m=0}^\infty  mp_m\vardel_{{p}_m} +\frac{1}{16}.
\end{equation}
One may recognize that $\hat{L}_{0}(e^\calK) = 0$ (which implies  $\hat{L}_{0}(\calK) = 1/24$),  is equivalent to the statement that  $\kappa_0$ on $\Moduli_g$ has degree $2g-2$ (Lemma \ref{ka0}).

\subsection {Bell polynomials}

We now focus on the $u^0$ coefficients of the operators $\hat{L}_n$, with $n>0$. As a preliminary step, we change variables to expressions of the form:
\begin{equation}
\sum_{i=0}^\infty i^d \hat{t}_{i}z^i = (z \vardel_z)^d \left( \sum_{i=0}^\infty \hat{t}_{i} z^i \right).
\end{equation}
Iterated applications of the operator $z \vardel_z$ to the right hand side of \eqref{schurcov} are described by a variation of the classical {\it Fa\`{a} di Bruno formula} \cite{faa,arbogast}, giving the result in terms of Bell polynomials \cite{b:poly}.

\begin{definition}\label{def:bell}
For $d \geq 0$, the {\bf $d$-th Bell polynomial} $B_d(x_1, \ldots, x_d)$ is defined by
\begin{equation}\label{bell}
\sum_{d=0}^\infty B_d(x_1, \ldots, x_d) \frac{y^n}{n!} = e^{\left( \sum_{j=1}^\infty x_j \frac{y^j}{j!}\right)}.
\end{equation}
\end{definition} 
The first few Bell polynomials are $B_0 = 1, B_1 = x_1, B_2 = x_2+x_1^2, B_3 = x_3+ 3x_2x_1+x_1^3, B_4 = x_4+ 4x_3x_1+3x_2^2+6 x_2x_1^2+ x_1^4$. The coefficient of a  monomial $\prod x_j^{m_j}$ in the Bell polynomial $B_n$ (with $n = \sum jm_j$) counts the number of ways to partition a set with $n$ elements into a collection of  $m_1$ unlabeled subsets of cardinality $1$, $m_2$ unlabeled subsets of cardinality $2$, etc.
 
\begin{lemma} \label{bell}
For all $i \geq 0$, let  $${\hat{q}}_i = \sum_{k=0}^\infty k^i \hat{p}_kz^k.$$
Then,
\begin{equation}\label{varfaa}
(z\vardel_z)^d e^{{\hat{q}}_0} = e^{{\hat{q}}_0} B_d({\hat{q}}_1, {\hat{q}}_2, \ldots, {\hat{q}}_d). 
\end{equation}
\end{lemma}
\begin{proof}
This statement is a ready adaptation of the classical {Fa\`{a} di Bruno formula}, which expresses the successive derivatives of a composite function $f(g(z))$ in terms of derivatives of $f$ and of $g$. In our case the function $f$ is an exponential function, hence all of its derivatives are equal to itself, and since we are replacing the derivation operator with $z\vardel_z$ the ${\hat{q}}_i$'s play the role of the successive derivatives of $g$.
While it is at this point an exercise to complete the proof this way, we also provide a brief sketch of a bijective proof, which we find more conceptually satisfactory. Note that $z\vardel_z {\hat{q}}_i = {\hat{q}}_{i+1}$, and that applying $z\vardel_z$ to $e^{{\hat{q}}_0} {\hat{q}}_{i_1}\ldots {\hat{q}}_{i_n}$ results in a sum of terms, where one term is obtained by multiplying $m$ by ${\hat{q}}_1$ (aka ``producing a new ${\hat{q}}_1$''), the other terms are obtained by raising by one one of the indices of one of the ${\hat{q}}_{i_j}'s$.
Now let us imagine performing successive applications of $z\vardel_z$ to $e^{{\hat{q}}_0}$ mantaining at all times all summands distinct, and let us call ``level $i$" the $i$-th application of the operator. 
We want to put in bijection the number of times a term of the form $e^{{\hat{q}}_0} {\hat{q}}_{i_1}\ldots {\hat{q}}_{i_n}$ appears (necessarily at level $d = i_1+\ldots i_n$) with the number of partitions of the set $[d]$ in $n$ subsets of cardinality $i_1$, \ldots $i_n$. For each variable ${\hat{q}}_{i_j}$, there are $i_j$ distinct levels where the index of such variable has been increased. We associate to such variable the subset of such levels, and by running over all variables that are appearing we obtain a partition of $[d]$. It is at this point easy to see that this construction realizes the bijection we desire.
We illustrate our proof in a simple example. Consider $(z\vardel_z)^3e^{{\hat{q}}_0} = 3e^{{\hat{q}}_0} {\hat{q}}_2{\hat{q}}_1+ \ldots$ and let us observe that $3$ is obtained as the ways to partition the set $[3]$ into two subsets of size $2$ and $1$. First we write out all levels up to $3$:
\begin{align*}
Level \ 1: &   e^{\hat{q}}_0 {\hat{q}}_1 \\
Level \ 2: &   e^{\hat{q}}_0 \bar{\hat{q}}_1 {\hat{q}}_1+  e^{\hat{q}}_0 {\hat{q}}_2 \\
Level \ 3: &   e^{\hat{q}}_0 \bar{\bar{\hat{q}}}_1\bar{\hat{q}}_1 {\hat{q}}_1+ e^{\hat{q}}_0 \bar{\hat{q}}_2 {\hat{q}}_1+e^{\hat{q}}_0 \bar{\hat{q}}_1 {\hat{q}}_2+  e^{\hat{q}}_0\bar{\bar{\hat{q}}}_1 {\hat{q}}_2+ e^{\hat{q}}_0 {\hat{q}}_3 \\
\end{align*}
The bars over the variables are purely combinatorial decorations that keep track of at which level a given variable appears. We then observe that the monomials that reduce to $e^{{\hat{q}}_0}{\hat{q}}_2{\hat{q}}_1$ when forgetting the decorations are in bijection with two part non-trivial partitions of $[3]$ where the singleton corresponds to the level where the variable ${\hat{q}}_1$ appears, and the two part subset corresponds to two levels: the level where the variable first appears as a ${\hat{q}}_1$ and the level where its index is raised by one.
\end{proof}

\begin{lemma}  \label{bellch}
With notation as in Lemma \ref{bell}, for any nonnegative integers $n, d$  we have:
\begin{equation}
\sum_{i=0}^\infty i^d \hat{t}_i{\vardel_{\hat{t}_{i+n}}} = \sum_{m=0}^\infty \left[B_d ({\hat{q}}_1, \ldots, {\hat{q}}_d)\right]_{z^m}\vardel_{\hat{p}_{m+n}}.
\end{equation}
\end{lemma}
\begin{proof}
In  analogy to the computations in Section \ref{covdil}, we have
\begin{align*}\label{genser2}
\sum_{i=0}^\infty i^d \hat{t}_{i}{\vardel_{\hat{t}_{i+n}}} &=\sum_{m=0}^\infty\left[(z\vardel_z)^d e^{\hat{q}_0}\left(\sum_{j=0}^\infty \frac{\vardel{\hat{p}_{j}}}{\vardel{\hat{t}_{i+n}}}z^{j-i-n}\right)\right]_{z^{m-n}}\vardel_{\hat{p}_m} \\
& = \sum_{m=0}^\infty \left[e^{{\hat{q}}_0} B_d ({\hat{q}}_1, \ldots, {\hat{q}}_d) e^{-{\hat{q}}_0}\right]_{z^{m-n}}\vardel_{\hat{p}_m},
\end{align*}
and the result follows from canceling the exponential terms and reindexing the summation.
\end{proof}

Lemma \ref{bellch}  allows to perform the change of variables for the $u^0$ coefficient of the operators $\hat{L}_n$.

\begin{lemma}\label{genusg}
For $n\geq 0$, let $ \left(i+\frac{3}{2}\right)\left(i+\frac{5}{2}\right)\cdots \left(i+\frac{2n+3}{2}\right)=\sum_{d=0}^{n+1} \alpha_{n,d}i^d$. Then:
\begin{align}
[\hat{L}_n]_{u^0} = \sum_{d=0}^{n+1} \alpha_{n,d} \left( \sum_{m=0}^\infty \left[B_d ({\hat{q}}_1, \ldots, {\hat{q}}_d)\right]_{z^m}\vardel_{\hat{p}_{m+n}} \right) .
\end{align}
\end{lemma}
 We now turn our attention to the $u^2$ coefficient. 

\subsection{ The $u^2$ coefficient: $n\geq 2$.}

We recall that from \eqref{parder} we can express:
\begin{equation}\label{firststep}
\vardel_{\hat{t}_{i}} = \sum_{j=0}^\infty \frac{\vardel \hat{p}_j}{\vardel \hat{t}_{i}} \vardel_{ \hat{p}_j}= e^{-\hat{p}_0}\sum_{j=0}^\infty  S_{j-i}(-\hat{\mathbf{p}})\vardel_{ \hat{p}_{j}}.
\end{equation}

It will be useful to take the partial derivative of \eqref{parder} with respect to a variable $\hat{p}_m$ to obtain:
 \begin{equation}\label{parder2}
\sum_{j=0}^\infty \frac{\vardel^2 \hat{p}_j}{\vardel \hat{p}_{m}\vardel \hat{t}_{i}} z^{j-i-m} = -e^{-\sum_{k=0}^\infty \hat{p}_{k} z^{k}}. 
\end{equation}

We may now use generating functions to perform the change of variables for an operator of second order:
\begin{align}\nonumber 
 \vardel_{\hat{t}_{l}} \vardel_{\hat{t}_{i}} &= \left( \sum_{m=0}^\infty \frac{\vardel \hat{p}_m}{\vardel \hat{t}_{l}} \vardel_{ \hat{p}_m}\right)\left( \sum_{j=0}^\infty \frac{\vardel \hat{p}_j}{\vardel \hat{t}_{i}} \vardel_{ \hat{p}_j}\right)\\\nonumber 
 & =  \sum_{j,m=0}^\infty \frac{\vardel \hat{p}_m}{\vardel \hat{t}_{l}} \frac{\vardel \hat{p}_j}{\vardel \hat{t}_{i}}  \vardel_{ \hat{p}_m} \vardel_{ \hat{p}_j}+ \sum_{m=0}^\infty \frac{\vardel \hat{p}_m}{\vardel \hat{t}_{l}}\sum_{j=0}^\infty \frac{\vardel^2 \hat{p}_j}{\vardel \hat{p}_{m}\vardel \hat{t}_{i}}\vardel_{ \hat{p}_j}\\\nonumber 
 & =  \sum_{j,m=0}^\infty e^{-2\hat{p}_0} S_{m-l}(-\hat{\mathbf{p}})  S_{j-i}(-\hat{\mathbf{p}}) \vardel_{ \hat{p}_m} \vardel_{ \hat{p}_j}+\\\nonumber 
 & + \sum_{j=0}^\infty    \left[\left(e^{-\sum_{k=0}^\infty \hat{p}_{k} z^{k}}\right)\left(-e^{-\sum_{k=0}^\infty \hat{p}_{k} z^{k}}\right)\right]_{z^{j-i-l}}  \vardel_{ \hat{p}_j}\\
 & =  e^{-2\hat{p}_0} \left( \sum_{j,m=0}^\infty S_{m-l}(-\hat{\mathbf{p}})  S_{j-i}(-\hat{\mathbf{p}}) \vardel_{ \hat{p}_m} \vardel_{ \hat{p}_j} - \sum_{j=0}^\infty    S_{j-i-l}(-2\hat{\mathbf{p}})  \vardel_{ \hat{p}_j}\right). \label{almost}
\end{align}

The last expression that needs to be addressed to complete the change of variables for $\hat{L}_n$, $n\geq 3$ is:
\begin{align}
  \sum_{i=0}^\infty \hat{t}_{i+1}  \vardel_{\hat{t}_{i}}\vardel_{\hat{t}_{n-2}}  &= \sum_{i=0}^\infty\left( \sum_{m=0}^\infty  \hat{t}_{i+1} \frac{\vardel \hat{p}_m}{\vardel \hat{t}_{i}}  \nonumber \vardel_{ \hat{p}_m}\right)\left( \sum_{j=0}^\infty \frac{\vardel \hat{p}_j}{\vardel \hat{t}_{n-2}} \vardel_{ \hat{p}_j}\right)\\ \nonumber 
  &=  \sum_{i,j,m=0}^\infty   \hat{t}_{i+1} \frac{\vardel \hat{p}_m}{\vardel \hat{t}_{i}} \frac{\vardel \hat{p}_j}{\vardel \hat{t}_{n-2}} \vardel_{ \hat{p}_m} \vardel_{ \hat{p}_j} + \sum_{i,m=0}^\infty \hat{t}_{i+1}\frac{\vardel \hat{p}_m}{\vardel \hat{t}_{i}}\sum_{j=0}^\infty  \frac{\vardel^2 \hat{p}_j}{\vardel \hat{p}_{m}\vardel \hat{t}_{n-2}}\vardel_{ \hat{p}_j}\\\nonumber 
  &=  \sum_{j,m=0}^\infty    \left[\left(e^{\sum_{k=0}^\infty \hat{p}_{k} z^{k}}- e^{\hat{p}_0}\right)\left(e^{-\sum_{k=0}^\infty \hat{p}_{k} z^{k}}\right)\right]_{z^{m+1}} e^{-\hat{p_0}} S_{j-n+2}(-\hat{\mathbf{p}}) \vardel_{ \hat{p}_m} \vardel_{ \hat{p}_j}\\\nonumber 
   &+  \sum_{j=0}^\infty    \left[\left(e^{\sum_{k=0}^\infty \hat{p}_{k} z^{k}}- e^{\hat{p}_0}\right)\left(e^{-\sum_{k=0}^\infty \hat{p}_{k} z^{k}}\right)\left(-e^{-\sum_{k=0}^\infty \hat{p}_{k} z^{k}}\right)\right]_{z^{j-n+3}} \vardel_{ \hat{p}_j}\\
 &=   e^{-\hat{p_0}}\left(\sum_{j,m=0}^\infty  - S_{m+1}(-\hat{\mathbf{p}}) S_{j-n+2}(-\hat{\mathbf{p}}) \vardel_{ \hat{p}_m} \vardel_{ \hat{p}_j}+  \sum_{j=0}^\infty \left(S_{j-n+3}(-2\hat{\mathbf{p}})- S_{j-n+3}(-\hat{\mathbf{p}}) \right)  \vardel_{ \hat{p}_j}\right). \label{lastpiece}
\end{align}

To deal with the case $n=2$, we compute:
\begin{align}
  \sum_{i=0}^\infty \hat{t}_{i+1}  \vardel_{\hat{t}_{i}} &= \sum_{i=0}^\infty\left( \sum_{m=0}^\infty  \hat{t}_{i+1} \frac{\vardel \hat{p}_m}{\vardel \hat{t}_{i}}  \nonumber \vardel_{ \hat{p}_m}\right)\\ \nonumber 
  &=  \sum_{m=0}^\infty    \left[\left(e^{\sum_{k=0}^\infty \hat{p}_{k} z^{k}}- e^{\hat{p}_0}\right)\left(e^{-\sum_{k=0}^\infty \hat{p}_{k} z^{k}}\right)\right]_{z^{m+1}}  \vardel_{ \hat{p}_m} \\
 &=  \sum_{m=0}^\infty  - S_{m+1}(-\hat{\mathbf{p}})  \vardel_{ \hat{p}_m}. \label{neqtwo}
\end{align}

Combining \eqref{firststep}, \eqref{almost}, \eqref{lastpiece}, \eqref{neqtwo} and appropriately reindexing we obtain the following.

\begin{lemma}\label{gminusonegen}
For $n\geq 2$
\begin{align}
[\hat{L}_n]_{u^2} = & -\frac{ e^{-2\hat{p}_0}}{2}\left(\sum_{i=0}^{n-1}\frac{(2i+1)!!(2n-2i-1)!!}{2^{n+1}}\right) \left(\sum_{j=0}^\infty    S_{j}(-2\hat{\mathbf{p}})  \vardel_{ \hat{p}_{j+n-3}}\right)+ \nonumber\\
 &+ \frac{ e^{-2\hat{p}_0}}{2}\left(\sum_{i=0}^{n-1}\frac{(2i+1)!!(2n-2i-1)!!}{2^{n+1}} \sum_{j,m=0}^\infty S_{m}(-\hat{\mathbf{p}})  S_{j}(-\hat{\mathbf{p}}) \vardel_{ \hat{p}_{m+n-2-i}} \vardel_{ \hat{p}_{j+i-1}}  \right).
\end{align}
\end{lemma}

\subsection{ The $u^2$ coefficient: $n=1$.}
In order to compute $\hat{L}_1$ we must perform the following computations:\\

\begin{align}
  \sum_{i=0}^\infty \hat{t}_{i+2}  \vardel_{\hat{t}_{i}} &= \sum_{i=0}^\infty\left( \sum_{m=0}^\infty  \hat{t}_{i+2} \frac{\vardel \hat{p}_m}{\vardel \hat{t}_{i}}   \vardel_{ \hat{p}_m}\right) =\nonumber\\ \nonumber 
  &=  \sum_{m=0}^\infty    \left[\left(e^{\sum_{k=0}^\infty \hat{p}_{k} z^{k}}- e^{\hat{p}_0}(1+\hat{p}_1z)\right)\left(e^{-\sum_{k=0}^\infty \hat{p}_{k} z^{k}}\right)\right]_{z^{m+2}}  \vardel_{ \hat{p}_m} \\
   &=\sum_{m=0}^\infty S_{1}(-\hat{\mathbf{p}}) S_{m+1}(-\hat{\mathbf{p}})  - S_{m+2}(-\hat{\mathbf{p}})  \vardel_{ \hat{p}_m}. \label{loneone}
  \end{align}

\begin{equation}
  \sum_{i=0}^\infty\frac{\hat{t}_{1}\hat{t}_{i+1}}{\hat{t}_{0}}  \vardel_{\hat{t}_{i}} =  \sum_{m=0}^\infty   S_{1}(-\hat{\mathbf{p}}) S_{m+1}(-\hat{\mathbf{p}})  \vardel_{ \hat{p}_m}.\label{lonetwo}
  \end{equation}

\begin{align}
  \sum_{i,j=0}^\infty \hat{t}_{i+1} \hat{t}_{j+1} \vardel_{\hat{t}_{i}}\vardel_{\hat{t}_{j}}  &= \sum_{i,j=0}^\infty \hat{t}_{i+1} \hat{t}_{j+1}\left( \sum_{l=0}^\infty  \frac{\vardel \hat{p}_l}{\vardel \hat{t}_{i}}  \nonumber \vardel_{ \hat{p}_l}\right)\left( \sum_{m=0}^\infty \frac{\vardel \hat{p}_m}{\vardel \hat{t}_{j}} \vardel_{ \hat{p}_m}\right)\\ \nonumber 
  &\hspace{-3cm}=  \sum_{i,j,l,m=0}^\infty   \hat{t}_{i+1} \hat{t}_{j+1}  \frac{\vardel \hat{p}_l}{\vardel \hat{t}_{i}} \frac{\vardel \hat{p}_m}{\vardel \hat{t}_{j}} \vardel_{ \hat{p}_l} \vardel_{ \hat{p}_m} + \sum_{i,j,l=0}^\infty \hat{t}_{i+1}\hat{t}_{j+1}\frac{\vardel \hat{p}_l}{\vardel \hat{t}_{i}}\sum_{m=0}^\infty  \frac{\vardel^2 \hat{p}_m}{\vardel \hat{p}_{l}\vardel \hat{t}_{j}}\vardel_{ \hat{p}_m}\\\nonumber 
  &\hspace{-3cm}=  \sum_{l,m=0}^\infty   \left[\left(e^{\sum_{k=0}^\infty \hat{p}_{k} z^{k}}- e^{\hat{p}_0}\right)\left(e^{-\sum_{k=0}^\infty \hat{p}_{k} z^{k}}\right)\right]_{z^{l+1}} 
    \left[\left(e^{\sum_{k=0}^\infty \hat{p}_{k} z^{k}}- e^{\hat{p}_0}\right)\left(e^{-\sum_{k=0}^\infty \hat{p}_{k} z^{k}}\right)\right]_{z^{m+1}}  \vardel_{ \hat{p}_l} \vardel_{ \hat{p}_m}\\\nonumber 
   &\hspace{-3cm}+  \sum_{m=0}^\infty    \left[\left(e^{\sum_{k=0}^\infty \hat{p}_{k} z^{k}}- e^{\hat{p}_0}\right)^2\left(e^{-\sum_{k=0}^\infty \hat{p}_{k} z^{k}}\right)\left(-e^{-\sum_{k=0}^\infty \hat{p}_{k} z^{k}}\right)\right]_{z^{m+2}} \vardel_{ \hat{p}_m}\\
 &\hspace{-3cm}=   \left(\sum_{l,m=0}^\infty   S_{l+1}(-\hat{\mathbf{p}}) S_{m+1}(-\hat{\mathbf{p}}) \vardel_{ \hat{p}_l} \vardel_{ \hat{p}_m}+  \sum_{m=0}^\infty \left(2S_{m+2}(-\hat{\mathbf{p}})- S_{m+2}(-2\hat{\mathbf{p}}) \right)  \vardel_{ \hat{p}_m}\right). \label{lone}
\end{align}

Combining the results of \eqref{loneone}, \eqref{lonetwo}, \eqref{lone}, we obtain the following.
\begin{lemma}\label{oneone}
\begin{align}
[\hat{L}_1]_{u^2} &=   \frac{e^{-2\hat{p}_0}}{8} \left( \sum_{m=0}^\infty \left(S_{m+2}(-\hat{\mathbf{p}})- S_{m+2}(-2\hat{\mathbf{p}})\right)  \vardel_{ \hat{p}_m}
  + \sum_{l,m=0}^\infty   S_{l+1}(-\hat{\mathbf{p}}) S_{m+1}(-\hat{\mathbf{p}}) \vardel_{ \hat{p}_l} \vardel_{ \hat{p}_m}\right).
\end{align}
\end{lemma}

\subsection{Proof of Theorem \ref{mainthm}}

The first part of Theorem \ref{mainthm} is proved by combining the results of Lemmas \ref{lem:ann}, \ref{genusg}, \ref{gminusonegen} and \ref{oneone} and switching back to the unhatted variables $p_k = -\hat{p}_k$. 
To prove that the recursions obtained from the vanishing of coefficients of $L_n(e^\calK)$ reconstruct $\calK$ from the ``unstable" term $1/24 p_0$, one observes that each $\hat{L}_n$ has a term of the form $A\vardel_{p_n}$, with $A\in \mathbb{Q}\smallsetminus 0$; this means that the vanishing of a coefficient of $L_n(e^\calK)$ compares the intersection number of a monomial $m$ containing $\kappa_n$ with a combination of other intersection numbers determined by the remaining terms of $\hat{L}_n$. A direct analysis of the remaining terms shows that the monomials that are compared to $m$ are either strictly shorter than $m$, or they correspond to intersection numbers on lower genus. This proves that any monomial can inductively be computed from the ``genus one" term  $1/24 p_0$.
 We illustrate this strategy for $g=2,3$ in the next section.
 



\section{Recursions for $\kappa$ Classes}\label{relations}
In this section we collect some of the relations among $\kappa$ classes that are produced by the vanishing of coefficients of $\hat{L}_n(e^\calK)$. We choose to exhibit a set of relations that inductively reconstruct all intersection numbers in genus $2$ and $3$.
Throughout this section we denote
$
[\kappa^I]_g:= \int_{\Moduli_g} \kappa^I. 
$
We extend this notation to the unstable term, and define 
$
[\kappa_0]_1 := \frac{1}{24}.
$
For $g\geq 2$, we only consider monomials with no factor of $\kappa_0$, since $[\kappa^I\kappa_0^n]_g = (2g-2)^n[\kappa^I]_g$. In genus $2$, we have the following:
$$
\begin{array}{rrl}
\left[\hat{L}_3 (e^\calK)\right]_1: &   [\kappa_3]_2 =&\frac{13}{630}[\kappa_0]_1+\frac{1}{210} [\kappa_0]_1^2 = \frac{1}{1152},\\ & & \\
\left[\hat{L}_1 (e^\calK)\right]_{p_2}: & [\kappa_2\kappa_1]_2 =&  \frac{48}{15} [\kappa_3]_2+ \frac{1}{30} [\kappa_0]_1 = \frac{1}{240},\\& & \\
\left[\hat{L}_1 (e^\calK)\right]_{p_1^2}: &  [\kappa_1^3]_2 =&  \frac{8}{3} [\kappa_2\kappa_1]_2-\frac{8}{15} [\kappa_3]_2+ \frac{1}{15}[\kappa_0]_1^2+\frac{1}{10}[\kappa_0]_1 = \frac{43}{2880}.\\
\end{array} 
$$

A set of reconstructing relations in genus $3$ is given by:
$$\hspace{-2cm}
\begin{array}{rrl}
\left[\hat{L}_6 (e^\calK)\right]_1: &  [\kappa_6]_3= &\frac{1}{99}[\kappa_3]_2+\frac{1}{1287}[\kappa_2\kappa_1]_2+\frac{1}{715}[\kappa_3]_2 [\kappa_0]_1,\\ & & \\
\left[\hat{L}_1 (e^\calK)\right]_{p_5}: &  [\kappa_5\kappa_1]_3=&12[\kappa_6]_3+\frac{1}{30}[\kappa_3]_2 ,\\ & & \\
\left[\hat{L}_2 (e^\calK)\right]_{p_4}: &   [\kappa_4\kappa_2]_3=&\frac{136}{7}[\kappa_6]_3+\frac{4}{35}[\kappa_3]_2+\frac{1}{35}[\kappa_3]_2[\kappa_0]_1,\\ & & \\
\left[\hat{L}_3 (e^\calK)\right]_{p_3}: &  

[\kappa_3^2]_3 = & \frac{136}{7}[\kappa_6]_3+ \frac{38}{315}[\kappa_3]_2+ \frac{1}{63}[\kappa_2\kappa_1]_2 - [\kappa_3]_2^2 + \frac{31}{630}[\kappa_3]_2[\kappa_0]_1  + \frac{1}{210}[\kappa_3]_2[\kappa_0]_1^2

 ,\\ & & \\
\left[\hat{L}_1 (e^\calK)\right]_{p_4p_1}: & [\kappa_4\kappa_1^2]_3= & -\frac{32}{15}[\kappa_6]_3+\frac{128}{15}[\kappa_5\kappa_1]_3+\frac{4}{3}[\kappa_4\kappa_2]_3+\frac{7}{30}[\kappa_3]_2+\frac{1}{30}[\kappa_2\kappa_1]_2+\frac{1}{15}[\kappa_3]_2[\kappa_0]_1 ,\\ & & \\
\left[\hat{L}_1 (e^\calK)\right]_{p_3p_2}: &  [\kappa_3\kappa_2\kappa_1]_3=&  -\frac{16}{5}[\kappa_6]_3 +\frac{28}{5}[\kappa_4\kappa_2]_3+\frac{16}{5}[\kappa_3^2]_3+\frac{1}{6}[\kappa_3]_2+\frac{1}{10}[\kappa_2\kappa_1]_2+\frac{16}{5}[\kappa_3]_2^2\\ & & \\ & &-[\kappa_3]_2[\kappa_2\kappa_1]_2+\frac{1}{30} [\kappa_3]_2[\kappa_0]_1,\\ & & \\
\left[\hat{L}_2 (e^\calK)\right]_{p_2^2}: &  [\kappa_2^3]_3=&-\frac{288}{35}[\kappa_6]_3+\frac{56}{5}[\kappa_4\kappa_2]_3+\frac{6}{35}[\kappa_3]_2+\frac{2}{7}[\kappa_2\kappa_1]_2+\frac{1}{35}[\kappa_3]_2[\kappa_0]_1+\frac{2}{35}[\kappa_2\kappa_1]_2[\kappa_0]_1 ,\\ & & \\
\left[\hat{L}_1 (e^\calK)\right]_{p_2^2p_1}: & [\kappa_2^2\kappa_1^2]_3= & -\frac{32}{15}[\kappa_5\kappa_1]_3-\frac{32}{15}[\kappa_4\kappa_2]_3+\frac{32}{5}[\kappa_3\kappa_2\kappa_1]_3+\frac{4}{3}[\kappa_2^3]_3+\frac{11}{30}[\kappa_3]_2+\frac{5}{6}[\kappa_2\kappa_1]_2 \\
& & \\
& &
+\frac{1}{15}[\kappa_1^3]_2+ \frac{32}{5}[\kappa_3]_2[\kappa_2\kappa_1]_2-2[\kappa_2\kappa_1]_2^2+\frac{1}{15} [\kappa_3]_2[\kappa_0]_1+\frac{1}{5}[\kappa_2\kappa_1]_2[\kappa_0]_1,\\ & & \\
\left[\hat{L}_1 (e^\calK)\right]_{p_3p_1^2}: &  [\kappa_3\kappa_1^3]_3=&-\frac{16}{5}[\kappa_5\kappa_1]_3-\frac{8}{15}[\kappa_3^2]_3 +\frac{28}{5}[\kappa_4\kappa_1^2]_3+\frac{8}{3}[\kappa_3\kappa_2\kappa_1]_3+\frac{29}{30}[\kappa_3]_2+\frac{8}{15}[\kappa_2\kappa_1]_2\\ & & \\
& &+\frac{1}{30}[\kappa_1^3]_2 -\frac{8}{15}[\kappa_3]_2^2+\frac{8}{3}[\kappa_3]_2[\kappa_2\kappa_1]_2-[\kappa_3]_2[\kappa_1^3]_2 +\frac{1}{2}[\kappa_3]_2[\kappa_0]_1 \\
\\ & & +\frac{2}{15}[\kappa_2\kappa_1]_2[\kappa_0]_1 +\frac{1}{15}[\kappa_3]_2[\kappa_0]_1^2, \\ & & \\
\left[\hat{L}_1 (e^\calK)\right]_{p_2p_1^3}: &  [\kappa_2\kappa_1^4]_3=&-\frac{16}{5}[\kappa_4\kappa_1^2]_3-\frac{8}{5}[\kappa_3\kappa_2\kappa_1]_3+\frac{16}{5}[\kappa_3\kappa_1^3]_3
+4[\kappa_2^2\kappa_1^2]_3+\frac{9}{10}[\kappa_3]_2+\frac{19}{5}[\kappa_2\kappa_1]_2 
\\ & & \\ &&
+\frac{29}{30}[\kappa_1^3]_2 -\frac{8}{5}[\kappa_3]_2[\kappa_2\kappa_1]_2+\frac{16}{5}[\kappa_3]_2[\kappa_1^3]_2+8[\kappa_2\kappa_1]_2^2-4[\kappa_2\kappa_1]_2[\kappa_1^3]_2
\\ & & \\ &&
+\frac{1}{5} [\kappa_3]_2[\kappa_0]_1+\frac{17}{10}[\kappa_2\kappa_1]_2[\kappa_0]_1+\frac{7}{30}[\kappa_1^3]_2[\kappa_0]_1+ \frac{1}{5}[\kappa_2\kappa_1]_2[\kappa_0]_1^2,\\ & & \\

\left[\hat{L}_1 (e^\calK)\right]_{p_1^5}: &  [\kappa_1^6]_3=&  -\frac{16}{3}[\kappa_3\kappa_1^3]_3 +
\frac{20}{3}[\kappa_2\kappa_1^4]_3+
\frac{17}{10}[\kappa_3]_2+\frac{35}{6}[\kappa_2\kappa_1]_2+12[\kappa_1^3]_2-\frac{16}{3}[\kappa_3]_2[\kappa_1^3]_2\\
& & \\
 & &  +\frac{80}{3}[\kappa_2\kappa_1]_2[\kappa_1^3]_2-10[\kappa_1^3]_2^2  +\frac{1}{3} [\kappa_3]_2[\kappa_0]_1+\frac{4}{3}[\kappa_2\kappa_1]_2[\kappa_0]_1+\frac{17}{3}[\kappa_1^3]_2[\kappa_0]_1. \\ & &  \\ & & +\frac{2}{3}[\kappa_1^3]_2[\kappa_0]_1^2
\end{array} 
$$




\bibliographystyle{amsalpha} 
\bibliography{biblio}
\end{document}